\documentclass[reqno,11pt,letterpaper]{amsart}

\usepackage{amsmath, amssymb, amsthm, amsfonts}
\usepackage{mathrsfs}
\usepackage[dvipsnames]{xcolor}
\usepackage{enumerate}
\usepackage{graphicx}
\usepackage{etoolbox}
\usepackage{ulem}
\usepackage{bm}
\usepackage{tikz}

\usepackage[T1]{fontenc}
\usepackage{lmodern} 

\newtheorem{theorem}{Theorem}[section]
\newtheorem{lemma}[theorem]{Lemma}

\newtheorem{remark}[theorem]{Remark}
\newtheorem{problem}[theorem]{Problem}

\theoremstyle{definition}

\numberwithin{equation}{section}

\usepackage[numbers]{natbib} 

\BeforeBeginEnvironment{thebibliography}{%
\small 
\setlength{\baselineskip}{1\baselineskip} 
}

\usepackage[colorlinks]{hyperref}

\newcommand{\lrb}[1]{\left( #1 \right)}

\def\p{\partial}
\def\no{\nonumber}

\def\vep{\varepsilon}

\def\ga{\gamma}
\def\Ga{\Gamma}

\def\al{\alpha}

\def\si{\sigma}

\def\ka{\kappa}
\def\de{\delta}

\def\lam{\lambda}

\def\mf{\mathbf}
\def\mr{\mathscr}

\def\u{\mathbf{u}}
\def\h{\mathbf{h}}

\def\f{\mathbf{f}}

\def\mlA{\mathcal{A}}
\def\mlB{\mathcal{B}}

\def\mlF{\mathcal{F}}

\def\mlQ{\mathcal{Q}}
\def\mlP{\mathcal{P}}

\def\mlU{\mathcal{U}}
\def\mlV{\mathcal{V}}

\def\mlG{\mathcal{G}}

\def\I{\mathcal{I}}

\def\tu{\tilde{u}}
\def\tc{\tilde{c}}

\def\tp{\tilde{p}}
\def\tP{\tilde{P}}

\def\tS{\tilde{S}}

\def\tB{\tilde{B}}
\def\tM{\tilde{M}}
\def\trho{\tilde{\rho}}

\def\tka{\tilde{\kappa}}

\def\bc{\bar{c}}

\def\bu{\bar{u}}
\def\bh{\bar{h}}

\def\bU{\bar{U}}

\def\bp{\bar{p}}
\def\bP{\bar{P}}

\def\bS{\bar{S}}
\def\bB{\bar{B}}
\def\bM{\bar{M}}
\def\bK{\bar{K}}

\def\brho{\bar{\rho}}

\def\bet{\bar{\eta}}

\def\bka{\bar{\kappa}}

\def\hu{\hat{u}}
\def\hB{\hat{B}}

\def\hS{\hat{S}}

\def\hka{\hat{\kappa}}

\def\du{\dot{u}}

\def\dg{\dot{g}}

\def\dW{\dot{W}}

\def\dP{\dot{P}}

\def\dS{\dot{S}}
\def\dN{\dot{N}}
\def\dB{\dot{B}}

\def\dka{\dot{\kappa}}
\def\deta{\dot{\eta}}

\def\dGa{\dot{\Gamma}}

\def\ddg{\ddot{g}}

\def\ddu{\ddot{u}}

\def\ddS{\ddot{S}}

\def\ddB{\ddot{B}}

\def\ddka{\ddot{\kappa}}
\def\ddeta{\ddot{\eta}}

\def\dd{\ddot}

\def\i{\mathfrak{i}}

\begin{document}

\title[The super-Alfv\'{e}nic shock] {Steady super-Alfv\'{e}nic  MHD shocks with aligned fields in  two-dimensional almost flat nozzles}

\author[S.K. Weng]{Shangkun Weng}
\address[Shangkun Weng]{\newline School of Mathematics and Statistics, Wuhan University, Wuhan, Hubei Province, 430072, People's Republic of China.}
\email{skweng@whu.edu.cn}

\author[W.G. Yang]{Wengang Yang}
\address[Wengang Yang]{\newline School of Mathematics and Statistics, Wuhan University, Wuhan, Hubei Province, 430072, People's Republic of China.}
\email{yangwg@whu.edu.cn}

\begin{abstract}
The Lorentz force induced by the magnetic field in MHD flow introduces a fundamental difference from pure gas dynamics by facilitating the anisotropic propagation of small disturbances, thus the type of steady MHD equations depends on not only the Mach number but also the Alfv\'{e}n number. In the super-Alfv\'{e}nic case, we derive an admissible condition for the locations of transonic shock fronts in terms of the nozzle wall profile and the exit total pressure (the kinetic plus magnetic pressure). Starting from this initial approximation, a nonlinear existence of super-Alfv\'{e}nic transonic shock solution to steady MHD equations is established. Our admissible condition is slightly different from that first introduced by Fang-Xin in [Comm. Pure Appl. Math., 74 (2021), pp. 1493–1544], and because our formulation is based on the deformation-curl decomposition of the steady MHD equations, our admissible condition has the advantage that a direct generalization to three dimensional case is available at least at the level of the initial approximation of the shock position. 


\end{abstract}

\keywords{steady MHD, structural stability, super Alfv\'{e}nic, transonic shock, hyperbolic-elliptic coupled, deformation-curl decomposition}
\subjclass[2020]{35L65, 35L67, 76N10, 76N15,76W05}
\date{}
\maketitle

\section{Introduction}\label{sec1}
The two-dimensional motion of a compressible, perfectly conducting inviscid fluid is governed by the steady compressible magnetohydrodynamic (MHD) equations: 
\begin{equation}\label{mhd0}
\begin{cases}
\p_{x_1}(\rho u_1)+ \p_{x_2} (\rho u_2)=0,\\
\p_{x_1}(\rho u_1^2 + p)+ \p_{x_2} (\rho u_1 u_2)=-h_2 (\p_{x_1} h_2 -\p_{x_2} h_1),\\
\p_{x_1}(\rho u_1 u_2 )+ \p_{x_2} (\rho u_2^2 +p)=h_1 (\p_{x_1} h_2 -\p_{x_2} h_1),\\
\p_{x_1}\left(\rho u_1 B+ h_2(u_1 h_2- u_2 h_1)\right)
+\p_{x_2}\left(\rho u_2 B-h_1(u_1 h_2- u_2 h_1)\right)=0,\\
\nabla_x(u_1 h_2- u_2 h_1)=0,\\
\p_{x_1} h_1 + \p_{x_2} h_2=0.
\end{cases}
\end{equation}
where $\rho$, $\u$, $p$, and $\h$ denote the density, velocity, pressure, and magnetic field, respectively. The quantity 
$$B = \frac{1}{2}|\u|^2+e + \frac{p}{\rho},$$
represents the specific total energy, with $e$ being the specific internal energy. The MHD equations provide a fundamental framework for modeling the dynamics of electrically conducting fluids, most notably plasmas, in the presence of magnetic fields\cite{BT2002,Fried1954}. A core aspect in the study of multidimensional hyperbolic conservation laws involves the analysis of fundamental waves, such as contact discontinuities, shocks, and rarefaction waves. Significant research has been directed toward understanding the stability and structure of vortex sheets \cite{ChenW2008,SunWZ2018} and contact discontinuities \cite{MTT2018,TrakWang2022,WangXin2024} for the unsteady MHD equations. In parallel,  steady-state solutions, characterized by time-independent flow, have also been examined in \cite{DT1950,Fried1954}. Here we consider the polytropic gases, the equation of state and the internal energy are
\begin{equation}\no
p=S\rho^{\ga},\,\,
e=\frac{p}{(\gamma-1)\rho}.
\end{equation} 
Here, $\gamma > 1$ denotes the adiabatic exponent, $S$ represents the entropy, and $s_0$ is a reference entropy constant. This paper aims to investigate the existence and stability of transonic shocks to \eqref{mhd0} in two-dimensional almost flat nozzles. 

In addition to the well-known dimensionless parameter, the Mach number $M^2=\frac{|\u|^2}{c^2(\rho,s)}$ which is defined as the ratio of the flow speed $|{\bf u}|$ to the local sound speed $c(\rho,s)=\sqrt{\p_\rho p(\rho,s)}$, there exists another key dimensionless quantity: the Alfv\'{e}n number. It is given by $A^2=\frac{|\u|^2}{c_a^2}$, where $c_a=\frac{|\h|^2}{\rho}$ denotes the Alfv\'{e}n wave speed. The acceleration of a nonconducting transonic flow in a de Laval nozzle results in a sonic transition ($M^2 = 1$) near the throat, where the governing equations change type from elliptic to hyperbolic. Due to the presence of a magnetic field, the flow is subjected to a Lorentz force, which provides a mechanism for the anisotropic propagation of small disturbances. This is the main difference between MHD and the pure gas dynamic flow. The types of the steady MHD equations may be either elliptic-hyperbolic or purely hyperbolic, depending not only on the Mach number, but also the Alfv\'{e}n number.  The acceleration of an infinitely conducting transonic flow with a velocity-aligned magnetic field must cross three transitions \cite{Chu1962,Grad1960}, specifically at $A^2 + M^2 = 1$, $A^2 = 1$, and $M^2 = 1$. The type of the governing differential equations changes accordingly: it is elliptic at low velocities, becomes hyperbolic after the first transition, reverts to elliptic after the second, and becomes hyperbolic again after the third. Consequently, the mathematical analysis of steady compressible MHD flows is significantly more complex than that of the Euler equations.

In the simplified but interesting case where the magnetic and velocity fields are aligned, the wealth of techniques developed in gas dynamics may also be applicable owing to the analogous mathematical structure. Specifically, we adopt the structural assumption:
\begin{equation}\label{colin_condi}
{\bf h}(x_1,x_2)= \kappa(x_1,x_2) (\rho {\bf u})(x_1,x_2),
\end{equation}
where $\kappa(x_1,x_2)$ represents a newly introduced scalar function that needs to be determined as part of the solution.

Substituting \eqref{colin_condi} into \eqref{mhd0}, we obtain
\begin{equation}\label{MHD}
\begin{cases}
\p_1(\rho u_1)+ \p_2 (\rho u_2)=0,\\
\p_1\left(\rho u_1^2(1-\rho\ka^2) + P\right)
+\p_2\left(\rho u_1 u_2 (1-\rho\ka^2)\right)=0,\\
\p_1\left(\rho u_1 u_2 (1-\rho \ka^2)\right)
+\p_2\left(\rho u_2^2(1-\rho\ka^2) + P\right)=0,\\
\rho(u_1\p_1+ u_2\p_2) B=0,\ \ \ \\
\rho(u_1\p_1+ u_2\p_2) \ka=0,
\end{cases}
\end{equation}
with the unknown vector $U=(u_1,u_2,p,S,\ka)$, where
\begin{equation*}
P=p+\frac{1}{2}|\h|^2=p+\frac{1}{2}\ka^2\rho^2(u_1^2+u_2^2)
\end{equation*}
denotes the total pressure, including both the kinetic and magnetic pressure.

Let the shock front be given by $x_1=\xi(x_2)$. The Rankine-Hugoniot jump conditions connect the states on the two sides of the shock:
\begin{align}
[\rho u_1]-\xi'(x_2)[\rho u_2]=0,\label{rh1}\\
[\rho u_1^2 + P-\ka^2\rho^2 u_1^2]
-\xi'(x_2)[(1-\ka^2\rho)\rho u_1 u_2]=0,\label{rh2}\\
[(1-\ka^2\rho)\rho u_1 u_2]
-\xi'(x_2)[\rho u_2^2 + P-\ka^2\rho^2u_2^2]=0,\label{rh3}\\
[B]=[\ka]=0.\label{rh4}
\end{align}



The background transonic shock is composed of two constant states:
\begin{equation*}
\bar{U}_\pm = (\bu_\pm,0,\bp_\pm,\bS_\pm, \bka_\pm),
\end{equation*}
which satisfy the jump conditions:
\begin{equation}
\begin{cases}
[\bar{\rho} \bu] = 0, \\
[\bar{\rho} \bu^2 + \bar{p}] = 0, \\
[\bar{B}] = [\bar{\kappa}] = 0,
\end{cases}
\end{equation}
with $[\bar{p}] = \bp_+ -\bp_- > 0$. It follows that
\begin{equation*}
\bB_+=\bB_-:=\bB,\,\bka_+=\bka_-:=\bka,
\end{equation*}
and the magnetic field for the background transonic flow is given by $\bh=\bka \brho_+\bu_+=\bka \brho_- \bu_-$. The corresponding Alfv\'{e}n numbers are defined as
\begin{equation}\label{def_alf_bk}
\bar{A}_\pm^2=\brho_\pm \frac{\bu_\pm^2}{|\bh|^2}=\frac{1}{\brho_\pm \bka^2}.
\end{equation}

To our purpose, we assume that the Alfv\'{e}n numbers of the background transonic shock states are greater than $1$, i.e. 
\begin{equation}\label{assump_bar_ka}
\bar{A}_{\pm}>1,\,\, \text{ or equivalently, } \,\,\bar{\kappa}^2<\frac{1}{\bar{\rho}_+}< \frac{1}{\bar{\rho}_-}.
\end{equation}
Then under \eqref{assump_bar_ka}, we should prescribe suitable boundary conditions at the entrance and exit of the nozzle to get a well-posed boundary value problem to investigating the stability of the shock solutions.

The two-dimensional nozzles to be studied in this article are given by
\begin{equation}\label{nozzle_x}
N :=\{(x_1,x_2):L_0\leq x_1\leq L_1,\, w_0(x_1)\leq 
x_2\leq w_1(x_1)\},
\end{equation}
so the nozzle walls are
\begin{equation}\label{wall_W_x}
W_{\i}:=
\{ 
(x_1,x_2)\,:\, x_2=w_\i(x_1),\,L_0\leq x_1\leq L_1
\},\,\i=0,1.
\end{equation}

Notice that the density $\rho$ and the pressure $p$ can be expressed by
\begin{equation}\label{def_rho_p}
\begin{split}
\rho&=\rho(B,S,|\u|^2)=\lrb{\frac{\ga-1}{\ga S}\lrb{B-\frac{1}{2}|\u|^2}}^{\frac{1}{\ga-1}},\\
p&=\rho(B,S,|\u|^2)=\lrb{\frac{\ga-1}{\ga S^\frac{1}{\ga}}\lrb{B-\frac{1}{2}|\u|^2}}^{\frac{\ga}{\ga-1}}.
\end{split}
\end{equation}

Employing these identities, we can rewrite system \eqref{MHD} as
\begin{equation}\label{defcurl_MHD}
\begin{cases}
(c^2-u_1^2)\p_1 u_1+(c^2-u_2^2)\p_2 u_2-u_1 u_2\lrb{\p_1 u_2+ \p_2 u_1} =0,\\
\p_1 (u_2(1-\rho \ka^2))-\p_2 (u_1(1-\rho \ka^2))\\
=-\frac{1}{u_1}\p_2 B+\frac{\rho^{\ga-1}}{(\ga-1)u_1}\p_2 S+\rho\ka\frac{|\u|^2}{u_1}\p_2\ka,\\
(u_1 \p_1 + u_2 \p_2 )(S,B,\ka)=0,
\end{cases}
\end{equation}
with unknown vector $\mlU=(u_1,u_2,S,B,\ka)$, where $c^2=c^2(B,|\u|^2)$ denotes the sound speed. The corresponding background solution is denoted by $\bar\mlU_\pm=(\bu_\pm,0,\bB,\bS_\pm,\bka)$.

We assume that
\begin{equation}\label{wall_eq_w_x}
w_0(x_1)\equiv 0,\,w_1(x_1)=1+ \si f(x_1),
\end{equation}
for some positive constant $\si$, and $f(x_1)\in C^{3,\al}([L_0,L_1])$ satisfying the following compatibility conditions:
\begin{equation}\label{comp_nozzle}
f^{(j)}(L_0)=0,j=0,1,2,3.
\end{equation}

Throughout this paper, the index $\i$ takes values in $\{0,1\}$. The entrance $\Ga_0$ and exit $\Ga_1$ of the nozzle $N$ are
\begin{equation}\label{En_Ex_x}
\Ga_\i:=\{(x_1,x_2):x_1=L_\i,\,0<x_2<w_1(L_\i)\}.
\end{equation}

The total pressure $P$ at the exit of the nozzle is prescribed by
\begin{equation}\label{end_pre_x} 
P(L_1,x_2)=\bP_+ +\si P_{ex}(x_2),
\end{equation}
where $\bP_+=\bp_++\frac{1}{2}\bka^2 \brho^2 \bu^2$ is the total pressure for the background solution, and $P_{ex}(x_2)\in C^{2,\al}(\Ga_1)$ satisfies
\begin{equation}
\|P_{ex}(x_2)\|_{2,\al;\Ga_1}<\infty
\end{equation}
for some $\al\in(0,1)$. 

\begin{figure}[ht]
\centering
\begin{tikzpicture}
\draw[->,thick] (-1.5,0)  -- (6.5,0) node (xaxis) [right] {$x_1$};
\draw[->,thick] (-1,-1) -- (-1,3.5) node (yaxis) [above] {$x_2$};

\draw[-,dashed] (0,2.5) -- (5.5,2.5);

\draw[-,thick] (0,0)node [below]{$x_1=L_0$} -- (0,2.5) ;
\draw[-,thick] (5.5,0)node [below]{$x_1=L_1$} -- (5.5,2.6) ;



\draw[->] (0.3,0.7)  --(2.4,0.7);
\draw[->] (0.3,1.7)  --(2.4,1.7);
\draw[-] (0.2,1.2)  --(0.2,1.2) node [right]{$Supersonic$};

\draw[->] (3.3,0.7)  --(5.3,0.7);
\draw[->] (3.3,1.7)  --(5.3,1.7);
\draw[-] (3.4,1.2)  --(3.4,1.2) node [right]{$Subsonic$};

\draw[-,dashed]  (2.9,2.6)--(2.9,0)  node [below]{$Shock$};

\draw[->] (3.5,2.9) node[above]{$x_2=1+\sigma f(x_1)$} --(3.2,2.6);

\draw[domain=0:5.5, smooth, thick] plot (\x, {0.04*(\x)*exp(-0.026*(\x-1)^2)*sin(4.7*\x r) +2.5});

\draw[domain=2.9:5.45,smooth,rotate around ={90:(2.9 ,0)}] plot(\x,{(0.06)*sin(6*\x r)});	

\end{tikzpicture}
\caption{Nozzle}
\end{figure}

This work aims to construct transonic shock solutions to the MHD equations \eqref{MHD}. We now present a detailed formulation of this problem.
\begin{problem}\label{prob1}(Transonic shock problem.) Suppose that \eqref{assump_bar_ka} holds, we investigate the existence of a shock solution $(\mlU_\pm(x);\xi(x_2))$ to the steady MHD equations \eqref{defcurl_MHD} in the nozzle $N$, which fulfills the following properties.
\begin{enumerate}

\item The location of the shock is given by
\begin{equation}\label{def_Ga_s}
\Ga_s:=\{(x_1,x_2):x_1=\xi(x_2),0<x_2<x_2^*\},
\end{equation}
where $(\xi(x_2^*),x_2^*)$ is the intersection point of the shock profile with the upper nozzle wall $W_1$.

\item In the domain $N_-$, ahead of the shock:
\begin{equation}\label{def_N-}
N_-:=\{(x_1,x_2):L_0<x_1<\xi(x_2),\,0<x_2<w_1(x_1)\},
\end{equation}
$\mlU_-$ solves the MHD equations \eqref{defcurl_MHD} with the boundary conditions:
\begin{equation}\label{bdy_U-_Ga0_Wall}
\begin{aligned}
\mlU_-&=\bar\mlU_- &&\text{on}\,\Ga_0\\
-u_1w_\i'+u_2&=0&&\text{on}\,W_{\i-},
\end{aligned}
\end{equation}
where $W_{\i-}=\p N_-\cap W_\i$.

\item Behind the shock, $\mlU_+$ satisfies the MHD equations \eqref{defcurl_MHD}, the slip boundary condition on the wall:
\begin{equation}\label{bdy_U+_wall}
-u_1w_\i'+u_2=0\quad\quad\text{on}\,W_{\i+},
\end{equation}
with $W_{\i+}=\p N_+\cap W_\i$, and the end pressure at the exit:
\begin{equation}\label{bdy_p+}
P_+(L_1,x_2)=\bP_+ +\si P_{ex}(x_2),
\end{equation}
in the domain $N_+$ defined as
\begin{equation}\label{def_N+}
N_+:=\{(x_1,x_2):\xi(x_2)<x_1<L_1,\,0<x_2<w_1(x_1)\}.
\end{equation}

\item The solutions $\mlU_-$ and $\mlU_+$ satisfy the RH conditions \eqref{rh1}--\eqref{rh4} on the shock front $\Gamma_s$.
\end{enumerate}
\end{problem}

\begin{remark} 
Let 
$$\mlU^{in}(x_2)=(u_1^{in},u_2^{in},B^{in},S^{in},\ka^{in})(x_2)\in C^{2,\al}(\Ga_0)$$
be given such that $\|U^{in}\|_{2,\al;\Ga_0}<\infty$. The analysis presented herein remains valid  for perturbed initial conditions of the form:
\begin{equation*}
\mlU_-=\bar\mlU_-+\si \mlU^{in}(x_2),\,\text{on}\,\Ga_0.
\end{equation*}
\end{remark}


\begin{remark}
If \eqref{assump_bar_ka} does not hold, then the situation is much more involved. If $\bar{A}_-^2<1$, then \eqref{defcurl_MHD} becomes elliptic-hyperbolic mixed in supersonic flow region. If $\bar{A}_+^2<1-\bar{M}_+^2$ (or $1-\bar{M}_+^2<\bar{A}_+^2<1$), then \eqref{defcurl_MHD} becomes elliptic-hyperbolic mixed (purely hyperbolic, respectively) in the subsonic flow region. Therefore, totally different types of boundary conditions at the entrance and exit should be prescribed accordingly. We leave these cases for future investigations.
\end{remark}

The transonic shock problem modeled by the steady Euler equations in nozzles is a fundamental yet challenging problem in mathematical fluid dynamics and has received significant attention in the literature, with a large body of work over the past few decades. The seminal works \cite{CF1948,EGM1984,L1982} provide early investigations of transonic shock phenomena through a quasi-one-dimensional model. Two types of transonic shock solutions typically serve as fundamental reference flows for the analysis of transonic shock problems within the framework of perturbation.

The first type comprises symmetric transonic shock solutions in a nozzle with a divergent section, exemplified by a radially symmetric shock in a two-dimensional angular sector or a spherically symmetric shock in a three-dimensional cone, in which the shock position can be uniquely determined from the exit pressure. The authors in \cite{LXY2009CMP,LXY2013} had proved the existence and stability of the transonic shock solution in divergent sectors under general perturbation of the nozzle wall and the exit pressure. Similar results for the axisymmetric perturbations of the nozzle shape and the pressure had been established in \cite{LXY2010JDE,WengXieXin2021}. The authors in \cite{LiuXY2016} proved the stability of spherically symmetric subsonic flows and transonic shocks in a spherical shell by requiring that the background transonic shock solutions satisfy some ``Structural Conditions". The authors in \cite{Weng25ArxivSph,WengX24ArxicCyl} removed the ``Structural Conditions" and established the existence and stability of cylindrical and spherical transonic shocks under three-dimensional perturbations of the incoming flows and the exit pressure, by developing the deformation-curl decomposition \cite{WengX2019}, an elaborate reformulation of the Rankine-Hugoniot conditions and introducing the ``spherical projection coordinates" to resolve the artificial singularities in the spherical coordinates. By reformulating the steady MHD system in terms of the deformation tensor and modified vorticity, the existence and structural stability of super-Alfv{\'e}nic cylindrical transonic shock solutions have been established under three-dimensional perturbations of the incoming flow and the exit total pressure \cite{WengYArxiv2025}. The main content of the deformation-curl decomposition is to decompose the steady Euler equations as three transport equations for the Bernoulli's quantity, the entropy and the first component of the vorticity, a deformation-curl system for the velocity field.



The second type of transonic shock solution is characterized by two constant states, with the shock position being arbitrary. The existence, uniqueness, and structural stability of these transonic shocks in nozzles were examined in \cite{ChenF2003,Chen2008,ChenYuan2008,XinYin2005} for multidimensional steady potential flow under a variety of boundary conditions.
Since the normal shock front in a flat nozzle can be located arbitrarily, there is no prior information to predict its position under wall perturbations. Therefore, a central difficulty in constructing transonic shock solutions is determining the shock front's location. Fang-Xin \cite{FangX2021} introduced a novel approach to establish the well-posedness of transonic shock solutions for the two-dimensional steady compressible Euler equations in an almost flat nozzle. 
See the recent generalizations to the axisymmetric flows \cite{FangG2021}, isothermal flows in a horizontal flat nozzle under vertical gravity \cite{FangG2022} and the two dimensional flow with a vertical magnetic field \cite{ZhaoZou2025}.  The key idea in \cite{FangX2021} is to first solve a free boundary problem for the linearized Euler system to get the most desirable information on the location of the shock front, then using its solution as an initial approximation and carrying out a further nonlinear iteration to establish the existence and stability of the transonic shock solution under perturbations to the nozzle wall and exit pressure. They crucially used the decomposition of 2-D steady Euler equations in the subsonic region as a first order elliptic system for the flow angle and the pressure, as well as transport equations for the entropy and Bernoulli's function, which was first observed in \cite{Serre1995} by Serre.

In contrast to \cite{FangX2021}, our approach utilizes a deformation-curl decomposition which was developed by Weng-Xin \cite{WengX2019} for the steady Euler equations. We adapt this decomposition to the steady MHD equations \eqref{MHD} in terms of the variables $(u_1,u_2,S,B,\ka)$, showing that the continuity equation is indeed a restriction equation on the deformation tensor for the velocity field, while the momentum equations give the equations for the curl of the velocity field. In the super-Alfv\'{e}nic case, we derive an admissible condition for the locations of transonic shock fronts in terms of the nozzle wall profile and the exit total pressure (the kinetic plus magnetic pressure). Starting from this initial approximation, a nonlinear existence of super-Alfv\'{e}nic transonic shock solution to steady MHD equations is established. Our admissible condition is slightly different from that first introduced by \cite{FangX2021}. Since the deformation-curl decomposition works well in the three dimensional setting, our admissible condition has the advantage that a direct generalization to three dimensional case is available at least at the level of the initial approximation of the shock position. For further information, one may refer to the main result Theorem \ref{main_thm} and the remarks thereafter.

The rest of the paper is structured as follows. Section \ref{sec2} reformulates the transonic shock problem by applying a Lagrangian transformation that flattens the streamlines and presents the main theorems of this paper. Section \ref{sec3} introduces a free boundary value problem for the linearized system around the background shock solution, aiming to determine the initial approximation of the shock front. Building on this approximation, Section \ref{sec4} develops a nonlinear iteration scheme, subsequently demonstrating that it is both well-defined and contractive, thus establishing the main result via the fixed point argument.

\section{Reformulation in Lagrangian Coordinates and Main Results} \label{sec2}
In this section, we introduce the Lagrangian transformation to straighten the nozzle 
wall, and reformulate the transonic shock problem \ref{prob1} into the Lagrangian coordinates. Define the coordinate mapping
\begin{equation}\label{mapping_xtoy}
(x_1, x_2)\mapsto (y_1, y_2)= \big(x_1, Y_2(x_1, x_2)\big),
\end{equation}
where the function $Y_2(x_1,x_2)$ satisfies the following system of ordinary differential equations:
\begin{equation}\label{Lag_transf_exp}
\begin{cases}
\frac{\p Y_2}{\p x_1}=-\rho_{-}u_{2-},\,
\frac{\p Y_2}{\p x_2}=\rho_{-}u_{1-},\,(x_1,x_2)\in N_-,\\
\frac{\p Y_2}{\p x_1}=-\rho_{+}u_{2+},\,
\frac{\p Y_2}{\p x_2}=\rho_{+}u_{1+},\,(x_1,x_2)\in N_+,\\
Y_2(L_0,0)=0,\,Y_2(L_1,0)=0.
\end{cases}
\end{equation}

It is clear that $Y_2(x_1,x_2)$ is well defined in $N_\pm$, respectively. It has been shown that the transformation $(x_1,x_2)\to(y_1,Y_2(x_1,x_2))$ is well-defined in the whole domain $N$ with Lipschitz continuity\cite{LXY2013}. The Jacobian determinant of this transformation is $\rho u_1$, nonzero in general. There exists an inverse mapping
\begin{equation*}
(y_1, y_2) \mapsto (x_1, x_2) = \big(y_1, X_2(y_1, y_2)\big).
\end{equation*}

Assuming that
\begin{equation}\label{nozzle_y}
\brho_-\bar{u}_-=\brho_+\bar{u}_+=1,
\end{equation}
under the mapping \eqref{mapping_xtoy}, the domain $N$ is changed into a rectangle:
\begin{equation*}
N^y=\{(y_1,y_2):L_0<y_1<L_1,\,0<y_2<1\}.
\end{equation*}

The entrance $\Ga_0$ and the exit $\Ga_1$ change to
\begin{equation}\label{En_EX_y}
\Ga_\i^y:=\{(y_1,y_2):y_1=L_\i,\,0<y_2<1\}.
\end{equation}

The nozzle walls were flattened as 
\begin{equation}\label{wall_W_y}
W_{\i}^y:=
\{ 
(y_1,y_2)\,:\, y_2=\i,\,L_0\leq y_1\leq L_1
\}.
\end{equation}

Define that
\begin{equation*}
(\tu_1,\tu_2, \trho, \tp, \tka)(y_1,y_2):=(u_1,u_2,\rho,p,\ka)(y_1,X_2(y_1,y_2)).
\end{equation*}
Then, the MHD equations \eqref{MHD} become to
\begin{equation}\label{eq_MHD_Lag}
\begin{cases}
\p_{y_1}\left(\frac{1}{\trho\tu_1}\right)-\p_{y_2}\left(\frac{\tu_2}{\tu_1}\right)=0,\\
\p_{y_1}\left(
\tu_1(1-\trho \tka^2)+\frac{\tP}{\trho\tu_1}
\right)
-\p_{y_2}\left(\frac{\tP \tu_2}{\tu_1}\right)=0,\\
\p_{y_1}\left((1-\trho \tka^2)\tu_2\right)
+\p_{y_2}\tP =0,\\
\p_{y_1} \tB=\p_{y_1} \tka=0,
\end{cases}
\end{equation}
where
\begin{equation*}
\tP=\tp+\frac{1}{2}\tka^2\trho^2(\tu_1^2+\tu_2^2).
\end{equation*}

The shock in the $y$-coordinates is represented by $y_1=\eta(y_2), y_2\in [0,1]$, thus the R-H conditions become
\begin{align}
\left[\frac{1}{\trho \tu_1}\right]+\eta'(y_2)\left[\frac{\tu_2}{\tu_1}\right]=0,\label{rh1_Lag}\\
\left[\tu_1(1-\trho \tka^2)+\frac{\tP}{\trho\tu_1}\right]
+\eta'(y_2)\left[\frac{\tP \tu_2}{\tu_1}\right]=0,\label{rh2_Lag}\\
[(1-\tka^2\trho) \tu_2]- \eta'(y_2)[\tP]=0\label{rh3_Lag},\\
[\tB]=[\tka]=0.\label{rh4_Lag}
\end{align}

The density $\trho$ and the pressure $\tp$ in the Lagrangian coordinates are changed into
\begin{equation}\label{exp_tilde_rho_p}
\begin{split}
\trho&=\trho(\tB,\tS,|\tilde\u|^2)=\lrb{\frac{\ga-1}{\ga \tS}\lrb{\tB-\frac{1}{2}|\tilde\u|^2}}^{\frac{1}{\ga-1}},\\
\tp&=\tp(\tB,\tS,|\tilde\u|^2)=\lrb{\frac{\ga-1}{\ga \tS^{\frac{1}{\ga}}}\lrb{\tB-\frac{1}{2}|\tilde\u|^2}}^{\frac{\ga}{\ga-1}}.
\end{split}
\end{equation}

It follows from \eqref{rh3_Lag} that
\begin{equation}\label{def_G0}
G_0(\mlU_+,\mlU_-):=[(1-\tka^2\trho) \tu_2]-\eta'(y_2)[\tP]=0.
\end{equation}

Utilizing \eqref{exp_tilde_rho_p} and substituting \eqref{def_G0} into in \eqref{rh1_Lag} and \eqref{rh2_Lag} yields 
\begin{align}
G_1(\mlU_+,\mlU_-)&:=\left[\frac{1}{\trho \tu_1}\right]+
\frac{[(1-\tka^2\trho) \tu_2]}{[\tP]}\left[\frac{\tu_2}{\tu_1}\right]=0,\label{def_G1}\\
G_2(\mlU_+,\mlU_-)&:=\left[\tu_1+\frac{\tP}{\trho\tu_1}-\tka^2\trho\tu_1\right]+
\frac{[(1-\tka^2\trho) \tu_2]}{[\tP]} \left[\frac{\tP \tu_2}{\tu_1}\right]=0.\label{def_G2}
\end{align}

Denoting
$$(\tu_1,\tu_2, \tS, \tB, \tka)(y_1,y_2):=(u_1,u_2,S,B,\ka)(y_1,X_2(y_1,y_2)),$$
$\tilde{\rho}$ and $\tilde{p}$ will be regarded as functions of $\tu_1,\tu_2, \tS, \tB$ as defined in \eqref{exp_tilde_rho_p}. Substituting \eqref{exp_tilde_rho_p} in to \eqref{defcurl_MHD}, then one has 
\begin{equation}\label{defcurl_sys_y}
\begin{cases}
(\tc^2-\tu_1^2)(\p_{y_1}\tu_1-\trho \tu_2 \p_{y_2}\tu_1) +(\tc^2-\tu_2^2)\trho \tu_1 \p_{y_2} \tu_2\\
-\tu_1 \tu_2\lrb{ \p_{y_1}\tu_2-\trho \tu_2 \p_{y_2}\tu_2+ \rho\tu_1 \p_{y_2} \tu_1} =0,\\
\p_{y_1}(\tu_2(1-\trho \tka^2))-\trho \tu_2 \p_{y_2}(\tu_2(1-\trho \tka^2))-\trho \tu_1\p_{y_2}(\tu_1(1-\trho \tka^2))\\
=-\trho\p_{y_2}\tB+\frac{\trho^{\ga}}{\ga-1}\p_{y_2}\tS+\tka|\tilde\u|^2\p_{y_2}\tka,\\
\p_{y_1} \tB=\p_{y_1} \tS=\p_{y_1} \tka=0.
\end{cases}
\end{equation}

If the initial condition for $\tB,\tka$ is prescribed as
\begin{equation*}
(\tB,\tka)(L_0,y_2)=(\bB,\bka),
\end{equation*}
then 
\begin{equation}\label{exp_tilde_B_ka}
(\tB_\pm,\tka_\pm)(y_1,y_2)\equiv(\bB,\bka),
\end{equation}
where we have used the last equation in \eqref{defcurl_sys_y} and \eqref{rh4_Lag}.

It follows from \eqref{defcurl_sys_y} that $(\tu_1,\tu_2,\tS)$ satisfies
\begin{equation}\label{eq_u1_u2}
\begin{cases}
(1-\tM_1^2)\p_{y_1}\tu_1-\tM_1 \tM_2 \p_{y_1}\tu_2-\trho \tu_2 \p_{y_2}\tu_1 +\trho \tu_1 \p_{y_2} \tu_2=0,\\
\trho \tka^2 \tM_1 \tM_2 \p_{y_1}\tu_1+(1-\trho \tka^2+\trho\tka^2 \tM_2^2)\p_{y_1}\tu_2\\
-\trho \tu_1 C(\tilde\mlU) \p_{y_2}\tu_1-\trho \tu_2 C(\tilde\mlU)\p_{y_2}\tu_2
=\trho^\ga\frac{1+\ga \trho \tka^2|\tilde{\mf M}|^2}{\ga-1}\p_{y_2}\tS,
\end{cases}
\end{equation}
and 
\begin{equation}\label{eq_S_y}
\p_{y_1} \tS=0,
\end{equation}
where
\begin{equation}\label{def_C_ka_U}
\begin{split}
\tM_i&=\tu_i/\tc,\,i=1,2,\,\tilde{\mf M}=(\tM_1,\tM_2),\\
C(\tilde\mlU):&=1-\trho\tka^2+\trho \tka^2\tM^2.
\end{split}
\end{equation}

In what follows, we will frequently drop the notation $\tilde{}$  when there is no danger of confusion. The equations in \eqref{eq_u1_u2} can be rewritten as
\begin{equation}
\mlA_1(\mlU)\,\p_{y_1}\u+\mlA_2(\mlU)\p_{y_2}\u=\f(\mlU)
\end{equation}
where $\f(\mlU)=(0,\trho^\ga\frac{1+\ga \trho \tka^2|\tilde{\mf M}|^2}{\ga-1}\p_{y_2}\tS)^\top$, and
\begin{equation}
\begin{split}
\mlA_1(\mlU)&=
\begin{pmatrix}
1-M_1^2 &-M_1 M_2\\
\rho \ka^2 M_1 M_2 & 1-\rho \ka^2+\rho\ka^2 M_2^2
\end{pmatrix},\\
\mlA_2(\mlU)&=
\begin{pmatrix}
-\rho u_2 & \rho u_1\\
-\rho u_1 C(\mlU) & -\rho u_2 C(\mlU)
\end{pmatrix}.
\end{split}
\end{equation}

The two solutions of the algebraic equations 
$$\det(\mlA_1(\mlU)-\lam \mlA_2(\mlU))=0,$$
can be explicitly given by
\begin{equation}
\lam_\pm(\mlU)=\frac{1}{\rho |\u|^2}
\left(	-u_2 \pm u_1\sqrt{\frac{(1-\rho\ka^2)(M^2-1)}{C(\mlU)}}\right).
\end{equation}

The classification of the system \eqref{defcurl_sys_y} hinges on its first two equations \eqref{eq_u1_u2}, leading to a dependence on the eigenvalues $\lambda_\pm(\mathcal{U})$.
Indeed, the eigenvalues $\lambda_\pm(\mathcal{U})$ are either two real numbers, corresponding to a hyperbolic system, or a complex conjugate pair, in which case the equations form a first-order elliptic system. Due to the magnetic field effects, the eigenvalues $\lambda_\pm(\mathcal{U})$ exhibit dependence on both the Mach number and the Alfv\'{e}n number, which significantly alter the type of the governing differential equations in both supersonic and subsonic flows. 

Under the assumption \eqref{assump_bar_ka}, namely, the super-Alfv\'{e}nic case $\bar{A}_\pm^2>1$, for the solutions $\mlU$ that are sufficiently close to the background flow $\bar\mlU$, we have
\begin{equation}
C(\mathcal{U}) > 0, 1-\rho\kappa^2 > 0,
\end{equation}
and find that the eigenvalues $\lambda_\pm(\mathcal{U})$ are real for supersonic flow, while a pair of conjugate complex numbers for subsonic flow. In other words, the upstream supersonic flow for system \eqref{defcurl_sys_y} is purely hyperbolic, while the downstream subsonic flow is of mixed elliptic-hyperbolic type in this case.

Let us now continue reformulating the transonic shock problem. The shock divides the domain $N^y$ into two subdomains:
\begin{equation}
\begin{split}
N_-^y:=\{(y_1,y_2):L_0<y_1<\eta(y_2),\,0<y_2<1\},\\
N_+^y:=\{(y_1,y_2):\eta(y_2)<y_1<L_1,\,0<y_2<1\}.
\end{split}
\end{equation}

The nozzle walls can be expressed as $W_{\i\pm}:=W_\i^y\cap \p N_{\pm}^y$. The transonic shock Problem \ref{prob1} is reformulated in Lagrangian coordinates as follows.

\begin{problem}\label{prob2}(Transonic shock problem in Lagrangian coordinates) Under the assumption \eqref{assump_bar_ka}, we look for a shock solution $(\mlU_\pm(y);\eta(y_2))$ to the MHD system \eqref{exp_tilde_B_ka}-\eqref{eq_S_y} in the nozzle $N^y$, which fulfills the following properties.
\begin{enumerate}

\item The location of shock is given by
\begin{equation}
\Ga_s^y:=\{(y_1,y_2):y_1=\eta(y_2),0<y_2<1\}.
\end{equation}

\item $\mlU_-$ solves the MHD equations \eqref{exp_tilde_B_ka}-\eqref{eq_S_y} with the boundary conditions:
\begin{equation}
\begin{aligned}
\mlU_-(L_0,y_2)&=\bar\mlU_- &&\text{on}\,\Ga_0^y,\\
u_{2-}(y_1,\i)&=\i \si f'(y_1) u_{1-}(y_1,\i) &&\text{on}\,W_{\i-}^y,
\end{aligned}
\end{equation}
in the supersonic domain $N_-^y$.

\item Behind the shock, $\mlU_+$ satisfies the MHD equations \eqref{exp_tilde_B_ka}-\eqref{eq_S_y}, the slip boundary condition on the wall:
\begin{equation}
u_{2+}(y_1,\i)=\i \si f'(y_1) u_{1+}(y_1,\i)\,\,\text{on}\,W_{\i+}^y,
\end{equation}
and the end pressure at the exit:
\begin{equation}\label{end_pre_y}
P_+=\bP_+ +\si \tP_{ex}(y_2),\quad\text{on}\,\Ga_1^y,
\end{equation}
where $\tP_{ex}(y_2)=P_{ex}(X_2(L_1,y_2))$ and
\begin{equation*}
X_2(L_1,y_2)=\int_0^{y_2}\frac{1}{(\rho u_1)(L_1,\tau)}\,d\tau.
\end{equation*}

\item The solution $\mlU_-$ and $\mlU_+$ satisfies the RH conditions \eqref{rh1_Lag}-\eqref{rh4_Lag} on the shock front $\Ga_s^y$.
\end{enumerate}
\end{problem}

Before stating our main theorem, we first define the following weighted H$\ddot{o}$lder space on a rectangle domain $\mlQ$:
\begin{equation*}
\mlQ:=\{(y_1,y_2):L_0<y_1<L_1,\,0<y_2<1\}.
\end{equation*}
For $x=(x_1,x_2),\,y=(y_1,y_2) \in \mlQ $,\, set 
\begin{equation*}
d_x=\min\{x_2,1-x_2\},\, d_{x,y}=\min\{d_x,d_y\}.
\end{equation*}
Define the space
\begin{equation*}
C_{m,\al}^{(\de)}(\mlQ):=
\{
u\in C^{m,\al}(\mlQ)\,:\,\| u \|_{m,\al;\mlQ}^{(\de)}\leq\infty
\},
\end{equation*}
where the norm $\|\cdot\|_{m,\al;\mlQ}^{(\de)}$ is defined as 
\begin{equation*}
\| u \|_{m,\al;\mlQ}^{(\de)}=
\sum_{j=0}^{m} [u]_{j,0;\mlQ}^{(\de)}+[u]_{m,\al;\mlQ}^{(\de)},
\end{equation*}
and 
\begin{equation*}
\begin{split}
[u]_{k,0;\mlQ}^{(\de)}&=
\sup_{|\beta|=k} d_x^{\max\{k+\de,0\} } |D^\beta u|,\,
k=0,1,2\cdots\,m,\\
[u]_{m,\al;\mlQ}^{(\de)}&=
\sup_{|\beta|=m} d_{x,y}^{\max\{m+\al+\de,0\} }
\frac{|D^\beta u(x)-D^\beta u(y)|}{|x-y|^\al},
\end{split}
\end{equation*}
for $\de\in \mathbb{R},\,m\in \mathbb{N}$ and $0<\al<1$.

The main results of this paper are stated below.
\begin{theorem}\label{main_thm}
Assume that the background magnetic field strength $\bka$ satisfies \eqref{assump_bar_ka}. Let $\bet^*\in(L_0,L_1)$ be such that
\begin{equation}\label{admi1}
\begin{split}
\bK_0\int_0^1 P_{ex}(y_2)\,dy_2=-\bK f(\bet^*)+\bu_+ f(L_1),
\end{split}
\end{equation}
and $f'(\bet^*)\neq 0$, where  
\begin{equation}\label{def_barK0_K}
\begin{split}
\bK_0&=\frac{1-\bM_+^2}{\brho^2 \bu^2 C(\bar\mlU_+)}>0,\ \ \ \  C(\bar\mlU_+)=1-\bar{\rho}_+ \bar{\kappa}^2 +\bar{\rho}_+ \bar{\kappa}^2 \bar{M}_+^2,\\
\bK&=\lrb{\frac{1-\bM_+^2}{C(\bar\mlU_+)}\lrb{\frac{1}{\ga \bM_+^2}+\brho_+ \bka^2} +1} \frac{\bu_+}{\bp_+}[\bp]>0.
\end{split}
\end{equation}
Then there exists a small constant $\si_*>0$, depending only on the background transonic solution, such that for any $0<\si<\si_*$, the Problem \ref{prob2} admits a transonic shock solution $(\mlU_-,\mlU_+;\eta)$. Moreover, the solution satisfies the following estimate:
\begin{equation}
\begin{split}
\|\mlU_- - \bar\mlU_-\|_{2,\al;N^y_-}+
\|\mlU_+-\bar\mlU_+\|_{1,\al;N_+^y}^{(-\al)}
+\|\eta'\|_{1,\al;N_+^y}^{(-\al)}
+|\eta(1)-\bet^*|
\leq C\si,
\end{split}
\end{equation}
where $C$ depends on the background flow.
\end{theorem}

\begin{remark}
If instead the upper nozzle wall is given by $w_1(x_2)= 1 + \int_{L_0}^{x_1} \tan(\sigma \Theta(s)) ds$, one may reformulate the steady MHD equations in terms of the flow angle and the total pressure as was done in \cite{FangX2021} for steady Euler equations. Then the initial approximate shock position $\bar{\eta}^*$ is thus determined by 
\begin{eqnarray}\label{admi2}
\bar{K}_0 \int_0^1 P_{ex}(\tau)d\tau=-\bar{K}\int_{L_0}^{\bar{\eta}^*}\Theta(\tau)d\tau+\bar{u}_+\int_{L_0}^{L_1}\Theta(\tau)d\tau. 
\end{eqnarray}
If $\Theta(\bar{\eta}^*)\neq 0$, one can also establish the existence of a super-Alfv\'{e}nic transonic shock solution to the Problem \ref{prob2} as in Theorem \ref{main_thm}.  When $\bar{\kappa}=0$, \eqref{admi2} reduces to the admissible condition that was proposed in \cite{FangX2021} for the steady Euler equations.
\end{remark}

\begin{remark}
The admissible condition \eqref{admi1} is slightly different from \eqref{admi2} that first proposed in \cite{FangX2021}. However, our admissible condition \eqref{admi1} has a natural generalization to the three-dimensional case. Consider the stability problem of the super-Alfv\'{e}nic transonic shock problem in a rectangle cylinder $(L_0,L_1)\times (0,1)\times (0,1)$ with the perturbations of the wall shape $x_3=1+\sigma f(x_1,x_2)$ and the exit total pressure $P(L_1,x_2,x_3)=\bar{P}_++ \sigma P_{ex}(x_2,x_3)$. Using the deformation-curl decomposition as in \cite{Weng25ArxivSph,WengX24ArxicCyl,WengZZ2025}, then the following admissible condition can be derived to determine the initial approximate position of the shock front:
\begin{equation}\label{admi3}
\bar{K}_0 \int_0^1\int_0^1 P_{ex}(x_2,x_3)dx_2 dx_3=-\bK \int_0^1 f(\bet^*,x_2) dx_2+\bu_+ \int_0^1 f(L_1,x_2) dx_2. 
\end{equation}
To further construct the solution to the shock problem, beside using the reformulation of the Rankine-Hugoniot jump condition and the novel technique for the first order deformation-curl system developed in \cite{Weng25ArxivSph,WengX24ArxicCyl}, one has to overcome the low regularity of the solution near the intersection of the shock front with the wall of the nozzle, which causes a serious issue that the streamline may not be uniquely determined. We leave it for future research.
\end{remark}


\section{The initial approximating locations of the shock front}\label{sec3}

Let $\bet(y_2) = \bet^*$, where $\bet^*$ is an unknown constant to be determined. The initial approximate position of the shock front is given by the free line:
\begin{equation*}
\dGa_s:=\{(y_1,y_2): y_1=\bet^*\}.
\end{equation*}
Then $\dGa_s$ divides the domain $N^y$ into two rectangle regions $\dN_-$ and $\dN_+$ as
\begin{equation*}
\begin{split}
\dN_-&=\{(y_1,y_2): L_0<y_1<\bet^*,\ 0<y_2<1\},\\
\dN_+&=\{(y_1,y_2): \bet^*<y_1<L_1,\ 0<y_2<1\}.
\end{split}
\end{equation*}

The nozzle walls, entrance, and exit remain unchanged and are denoted by $\dW_\i=W_\i^y$, $\dGa_\i=\Ga_i^y$, respectively. Additionally, we define $\dW_{\i\pm}:=\p\dN_{\pm}\cap W_\i^y$.

The fluctuation is denoted as $\dot\mlU_\pm=(\du_{1\pm},\du_{2\pm},\dS_{\pm},\dB,\dka)$. The free surface will be determined simultaneously with the linear shock solution $(\dot\mlU_\pm;\deta')$. Linearizing the Rankine-Hugoniot conditions \eqref{def_G0}–\eqref{def_G2} around the reference flow on the free surface yields:
\begin{align}
G_0(\mlU_+,\mlU_-)&=\beta_0^+\cdot \dot\mlU_+ + \beta_0^-\cdot \dot\mlU_- -\deta' [\bp]+O(|\dot\mlU_\pm|^2),\,\,&&\text{on}\,\,\dGa_s,\label{lin_G0}\\
G_j(\mlU_+,\mlU_-)&=\beta_{j}^+\cdot \dot\mlU_+ + \beta_{j}^-\cdot \dot\mlU_-+O(|\dot\mlU_\pm|^2),\,j=1,2,\,&&\text{on}\,\,\dGa_s,\label{lin_G12}
\end{align}
where the coefficients $\beta_j^\pm,\,j=0,1,2$, are explicitly given by
\begin{align*}
\beta_0^\pm&=\pm\lrb{0,1-\brho_\pm\bka^2,0,0,0}^\top,\\
\beta_1^\pm&=\pm\frac{1}{\brho \bu}\lrb{\frac{\bM_\pm^2-1}{\bu_\pm},0,\frac{1}{(\ga-1) \bS_\pm},-\frac{1}{\bc_\pm^2}, 0   }^\top,\\
\beta_2^\pm&=\pm \lrb{\bar{d}_\pm \frac{\bM_\pm^2-1}{ \bM_\pm^2},0,
\frac{\brho_\pm \bu_\pm \bka^2}{2(\ga-1)\bS_\pm},
\frac{1-\bar{d}_\pm}{\bu_+},-\brho_\pm \bu_\pm \bka}^\top,
\end{align*}
where
\begin{equation*}
\bar{d}_\pm =\frac{1}{\ga}+\frac{1}{2}\brho_\pm\bka^2\bM_\pm^2.
\end{equation*}

The remainder of this section focuses on constructing the linear solutions $(\dot\mlU_\pm;\deta')$ together with the free interface $\dGa_s$, which will serve as an initial approximation for the nonlinear shock problem.

\subsection{\texorpdfstring{The solution $\dot\mlU _-$ in $N^y$}{}} The fluctuation $\dot\mlU_-$ satisfies the linearized MHD system around the uniform supersonic reference flow $\bar\mlU_-$:
\begin{equation}\label{eq_dot_U-}
\begin{cases}
(1-\bM_-^2)\p_{y_1}\du_{1-} +\brho \bu \p_{y_2} \du_{2-}=0,\\
(1-\brho_- \bka^2)\p_{y_1}\du_{2-}
-\brho \bu C(\bar\mlU_-) \p_{y_2}\du_{1-}
=\brho_-^\ga \frac{1+\ga \brho_-\bka^2 \bM_-^2}{\ga-1}\p_{y_2}\dS_-,\\
\p_{y_1}\dS_-=\p_{y_1}\dB_-=\p_{y_1}\dka_-=0.
\end{cases}
\end{equation}

The equations \eqref{eq_dot_U-} subject to the boundary conditions:
\begin{equation}\label{bd_U-_Ga0}
\dot\mlU_-(L_0,y_2)=0,\,\, 0<y_2<1,
\end{equation}
and
\begin{equation}\label{bd_U-_wall}
\du_{2-}(y_1,\i)=\i \si \bu_-f'(y_1) ,\,\,L_0<y_1<\bet^*.
\end{equation}

Since \eqref{eq_dot_U-} is a hyperbolic system with constant coefficients in the supersonic domain,  the application of the classical theory results in the following lemma.

\begin{lemma}\label{lem_dot_U-} Assume that \eqref{comp_nozzle} holds. There exists a unique solution $\dot\mlU_-$ to \eqref{eq_dot_U-}–\eqref{bd_U-_wall} satisfying the estimate
\begin{equation}\label{est_dot_U-}
\|\dot\mlU_-\|_{2,\al;N^y} \leq C \|\si \bu_-  f'\|_{2,\al;W_1} \leq C(\bar\mlU_-, L_0, L_1) \sigma,
\end{equation}
where $C(\bar\mlU_-, L_0, L_1)$ depends only on $\bar\mlU_-$, $L_0$, and $L_1$.

Furthermore, the solution satisfies the following properties:
\begin{equation}\label{qP_sk=0}
\dS_-=\dB_-=\dka_-=0,
\end{equation}
and the compatibility condition:
\begin{equation}\label{comp_dot_U-}
\frac{1-\bM_-^2}{\brho \bu} \int_0^1\du_{1-}(y_1,y_2)\,dy_2=-\si \bu_- f(y_1).
\end{equation}
\end{lemma}
\begin{proof}
It suffices to show \eqref{comp_dot_U-}. The first equation in \eqref{eq_dot_U-} suggests that there is a potential function $\phi(y_1,y_2)$ such that
\begin{equation*}
\p_{y_1}\phi=-\brho \bu \du_{2-},\,
\p_{y_2}\phi=(1-\bM_-^2) \du_{1-},\,
\phi(L_0,0)=0.
\end{equation*}

By initial conditions, we have $\p_{y_2}\phi(L_0,y_2)\equiv0$ and $\phi(L_0,y_2)\equiv 0$.

Utilizing the boundary conditions \eqref{bd_U-_Ga0} and \eqref{bd_U-_wall}, we obtain
\begin{equation*}
\begin{split}
&\frac{1-\bM_-^2}{\brho \bu} \int_0^1\du_{1-}(y_1,y_2)\,dy_2
=\frac{1}{\brho \bu}\int_0^1\p_{y_2}\phi(y_1,y_2)\,dy_2\\
=&\frac{\phi(y_1,1)-\phi(y_1,0)}{\brho \bu}=\frac{1}{\brho \bu}
\int_{L_0}^{y_1}\p_{y_1}\phi(\tau,1)-\p_{y_1}\phi(\tau,0)\,d \tau\\
=&-\int_{L_0}^{y_1}\du_{2-}(\tau,1)d\tau=-\si  \bu_-\int_{L_0}^{y_1} f'(\tau)d\tau=-\si \bu_- f(y_1),
\end{split}
\end{equation*}
which is exactly \eqref{comp_dot_U-}.
\end{proof}

\subsection{\texorpdfstring{The determination of $\dot\mlU_+$ and $\bet^*$}{}}
The fluctuation $\dot\mlU_+$ satisfies the linearized MHD system:
\begin{equation}\label{eq_dot_U+}
\begin{cases}
(1-\bM_+^2)\p_{y_1}\du_{1+} +\brho \bu \p_{y_2} \du_{2+}=0,\\
(1-\brho_+ \bka^2)\p_{y_1}\du_{2+}-\brho \bu C(\bar\mlU_+) \p_{y_2}\du_{1+}=\dot{f}_2,\\
\p_{y_1}\dS_+=\p_{y_1}\dB_+=\p_{y_1}\dka_+=0,
\end{cases}
\end{equation}
where $\dot{f}_2=\brho_+^\ga \frac{1+\ga \brho_+\bka^2\bM_+^2}{\ga-1}\p_{y_2}\dS_+$.

On the nozzle walls, the equations \eqref{eq_dot_U+} subject to the boundary conditions:
\begin{equation}\label{bd_U+_wall}
\du_{2+}(y_1,\i)=\i \si \bu_+ f'(y_1) ,\,\, \bet^*<y_1<L_1.
\end{equation}


With the supersonic state $\dot\mlU_-$ being specified, the boundary conditions for $\dot\mlU_+$ at the free interface $y_1=\bet^*$ are completely determined by the linearized Rankine-Hugoniot conditions \eqref{lin_G12} and the upstream flow. Specifically, we impose the following boundary conditions on the free interface $y_1=\bet^*$:
\begin{equation}\label{bdy_lin_Rh_dot}
\beta_j^+\cdot\dot\mlU_+(\bet^*,y_2) + \beta_j^-\cdot \dot\mlU_-(\bet^*,y_2)=0,\,j=1,2.
\end{equation}
These relations ultimately determine the boundary values of $\du_{1+}$ and $\dS_+$ at $y_1=\bet^*$.

\begin{lemma}\label{RH_bd_condi}
On the free interface $y_1=\bet^*$, there holds 
\begin{equation}\label{bd_dot_u1_S}
\begin{cases}
\du_{1+}(\bet^*,y_2)=b_u^s \du_{1-}(\bet^*,y_2),\\
\dS_+(\bet^*,y_2)=b_s^s \du_{1-}(\bet^*,y_2),
\end{cases}
\end{equation}
where 
\begin{equation}\label{def_bus_bss}
\begin{split}
b_u^s&=\frac{\bM_+^2}{\bM_-^2}\frac{\bM_-^2-1}{\bM_+^2-1}=\frac{\bar{\rho}_-\bar{p}_-}{\bar{\rho}_+\bar{p}_+}\frac{\bM_-^2-1}{\bM_+^2-1},\\
b_s^s&=\frac{(\ga-1)\bS_+}{\bp_+\bu_-}(\bM_-^2-1)[\bp].
\end{split}
\end{equation}
\end{lemma}

\begin{proof}
The R-H conditions \eqref{rh4_Lag} and $\p_{y_1}\dB_+=\p_{y_1}\dka_+=0$ imply that
\begin{equation}\label{exp_dotB_ka}
\begin{split}
& \dB_+(y_1,y_2)=\dB_+(\bet^*,y_2)=\dB_-(\bet^*,y_2)=0,\\
&\dka_+(y_1,y_2)=\dka_+(\bet^*,y_2)=\dka_-(\bet^*,y_2)=0.
\end{split}
\end{equation}

Denoting $\dot\mlV_\pm=( \du_{1\pm}, \dS_\pm)^\top$, then \eqref{bdy_lin_Rh_dot} simplifies to
\begin{equation}
\bar\mlB_{s+}  \dot\mlV_+(\bet^*,y_2)
=\bar\mlB_{s-}  \dot\mlV_-(\bet^*,y_2) ,
\end{equation}
where  $\bar\mlB_{s\pm}=(\bar{b}_{1\pm},\bar{b}_{2\pm})$ is a $2\times 2$ matrix and $\bar{b}_{1\pm},\bar{b}_{2\pm}$ are column vectors defined as
\begin{equation*}
\bar{b}_{1\pm}=(\frac{\bM_\pm^2-1}{\bu_\pm},\bar{d}_\pm \frac{\bM_\pm^2-1}{\bM_\pm^2})^\top,
\bar{b}_{2\pm}=\frac{1}{(\ga-1)\bS_\pm}(1,\frac{1}{2}\brho \bu \bka^2)^\top
\end{equation*}
Noting that
\begin{equation}
\det{\bar\mlB_{s\pm}}=\frac{1-\bM_\pm^2}{\ga(\ga-1)\bM_\pm^2 \bS_\pm^2}\neq0,
\end{equation}
we obtain 
\begin{equation*}
\dot\mlV_+(\bet^*,y_2)=\bar\mlB_{s+}^{-1}\bar\mlB_{s-} \dot\mlV_-(\bet^*,y_2).
\end{equation*}

Since $\dS_-\equiv0$, a direct computation shows that
\begin{equation*}
\dot\mlV_+(\bet^*,y_2)=\bar\mlB_{s+}^{-1}\bar\mlB_{s-} \dot\mlV_-(\bet^*,y_2)=(b_u^s,b_s^s)^\top \du_{1-}(\bet^*,y_2),
\end{equation*}
which completes the proof.
\end{proof}


At the exit of the nozzle, the total pressure is prescribed by
\begin{equation}
P_+(L_1,y_2)=\bP_+ +\si P_{ex}(y_2).
\end{equation}

Clearly, the total pressure $P$ can also be expressed as a function of $\mlU$. Defining $P = \mlP(\mlU)$, one has 
\begin{equation}\label{def_mlP}
\begin{split}
\mlP(\mlU)&=\bP-\brho \bu C(\bar\mlU)\du_{1}-\frac{\bp +\bka^2\brho^2\bu^2}{(\ga-1)\bS}\dS\\
&+\brho(1+\brho \bka^2 \bM^2)\dB +\bka \brho^2 \bu^2 \dka +O(|\dot\mlU|^2).
\end{split}
\end{equation}

Therefore, the leading order satisfies
\begin{equation*}
\si P_{ex}(y_2)=-\brho \bu C(\bar\mlU_+)\du_{1+}(L_1,y_2)-\frac{\bp_+ +\bka^2\brho^2\bu^2}{(\ga-1)\bS_+}b_s^s \du_{1-}(\bet^*,y_2),
\end{equation*}
at the exit of the nozzle.

It remains to determine $\bet^*$ and $(\du_{1+},\du_{2+},\dS_+)$ via the following boundary value problems:
\begin{equation}\label{BVP_dotS+}
\begin{cases}
\p_{y_1}\dS_+=0,\,\text{in }\dN_+,\\  
\dS_{+}(\bet^*,y_2)=b_s^s \du_{1-}(\bet^*,y_2):=\dg_s(y_2),\,\text{on }\dGa_s,
\end{cases}
\end{equation}
and
\begin{equation}\label{BVP_dot_u1_u2}
\begin{cases}
(1-\bM_+^2)\p_{y_1}\du_{1+} +\brho \bu \p_{y_2} \du_{2+}=0,\,\text{in }\dN_+,\\
(1-\brho_+ \bka^2)\p_{y_1}\du_{2+}-\brho \bu C(\bar\mlU_+) \p_{y_2}\du_{1+}=\dot{f}_2,\,\text{in }\dN_+,\\
\du_{1+}(\bet^*,y_2)=b_u^s \du_{1-}(\bet^*,y_2):=\dg_1(y_2),\,\text{on }\dGa_s,\\
\du_{1+}(L_1,y_2)=b_u^1\du_{1-}(\bet^*,y_2)-\si \frac{P_{ex}(y_2)}{\brho \bu C(\bar\mlU_+)}:=\dg_3(y_2),\,\text{on }\dGa_1,\\
\du_{2+}(y_1,\i)=\i \si \bu_+ f'(y_1):=\dg_{2\i+2}(y_1),\,\text{on }\dW_{\i+},
\end{cases}
\end{equation}
with $b_u^1$ given by
\begin{equation}
\begin{split}
b_u^1&=-\frac{(\bp_+ +\bka^2\brho^2\bu^2)}{\brho \bu C(\bar\mlU_+)} \frac{\bM_-^2-1}{\bp_+\bu_-}[\bp].
\end{split}
\end{equation}



We first establish the solvability condition for the problem \eqref{BVP_dot_u1_u2}.
\begin{lemma}
Suppose that $\bka$ satisfies \eqref{assump_bar_ka}. Given $\bet^*\in(L_0,L_1)$, there exists a unique solution $(\du_{1+},\du_{2+})$ to the boundary value problem \eqref{BVP_dot_u1_u2} if and only if 
\begin{equation}\label{comp_condi}
\bK_0\int_0^1 P_{ex}(y_2)\,dy_2=-\bK f(\bet^*)+\bu_+ f(L_1),
\end{equation}
where $\bK_0,\bK$ are defined in \eqref{def_barK0_K}.
\end{lemma}
\begin{proof}
Integrating the first equation in \eqref{BVP_dot_u1_u2} over $\dN_+$ results in the following necessary solvability condition:
\begin{equation*}
\begin{split}
\int_{\bet^*}^{L_1}\int_{0}^{1}\p_{y_2}\du_{2+}(y_1,y_2)\,dy_1dy_2
=-\frac{1-\bM_+^2}{\brho \bu} \int_{\bet^*}^{L_1}\int_{0}^{1}\p_{y_1}\du_{1+}(y_1,y_2)\,dy_1dy_2.
\end{split}
\end{equation*}

Integrating by parts yields
\begin{equation}
\begin{split}
\frac{1-\bM_+^2}{\brho \bu}
\int_0^1(\dg_1-\dg_3)(y_2)dy_2=\int_{\bet*}^{L_1} \dg_4(y_1) \,dy_1
=\si \bu_+\int_{\bet*}^{L_1} f'(\tau)d\tau.
\end{split}
\end{equation}

By applying the boundary condition of  $\du_{1+}$ on $\dGa_s$ and $\dGa_1$, we have
\begin{equation*}
\begin{split}
&\frac{1-\bM_+^2}{\brho \bu}\int_0^1 \du_{1+}(\bet*,y_2)dy_2
=\frac{1-\bM_+^2}{\brho \bu} b_u^s\int_0^1  \du_{1-}(\bet^*,y_2)\,dy_2\\
=&-\si \frac{\bM_+^2}{\bM_-^2} \bu_-f(\bet^*)=-\si \bu_+\frac{\bp_-}{\bp_+} f(\bet^*),
\end{split}
\end{equation*}
and 
\begin{equation*}
\begin{split}
&\frac{1-\bM_+^2}{\brho \bu}\int_0^1 \du_{1+}(L_1,y_2)dy_2 \\
=&\frac{1-\bM_+^2}{\brho \bu} b_u^1\int_0^1 \du_{1-}(\bet^*,y_2)dy_2 
-\frac{1-\bM_+^2}{\brho^2 \bu^2 C(\bar\mlU_+)}\si \int_0^1 P_{ex}(y_2)\,dy_2 \\
=&-(1-\bM_+^2)\frac{(\bp_+ +\bka^2\brho^2\bu^2)}{\brho \bu C(\bar\mlU_+)} \frac{[\bp]}{\bp_+}\si f(\bet^*)
-\frac{1-\bM_+^2}{\brho^2 \bu^2 C(\bar\mlU_+)}\si\int_0^1 P_{ex}(y_2)\,dy_2,
\end{split}
\end{equation*}
where we have used \eqref{comp_dot_U-}.

Consequently, the solvability condition transforms into
\begin{equation*}
\begin{split}
\frac{1-\bM_+^2}{\brho^2 \bu^2 C(\bar\mlU_+)}\int_0^1 P_{ex}(y_2)\,dy_2 =&-(1-\bM_+^2)\frac{(\bp_+ +\bka^2\brho^2\bu^2)}{\brho \bu C(\bar\mlU_+)} \frac{[\bp]}{\bp_+} f(\bet^*)\\
&+ \bu_+\frac{\bp_-}{\bp_+} f(\bet^*)+ \bu_+(f(L_1)-f(\bet^*)),
\end{split}
\end{equation*}
which coincides with \eqref{comp_condi}.

\end{proof}

Observing that $\bK_0\neq0$, we can define 
\begin{equation}
F(\eta)=-\frac{\bK}{\bK_0}f(\eta)+\frac{\bu_+}{\bK_0}f(L_1)
\end{equation}
and
\begin{equation*}
\underline{F}:=\inf_{\eta\in(L_0,L_1)}F(\eta),\,\overline{F}:=\sup_{\eta\in(L_0,L_1)}F(\eta).
\end{equation*}

Since $F(\eta)$ is continuous in $\eta$, the intermediate value theorem ensures that for any $P_{ex}(y_2)$ with
\begin{equation}\label{compa_Pe}
\underline{F}\leq \int_0^1 P_{ex}(y_2)dy_2\leq \overline{F},
\end{equation}
there exists $\bet^*$ such that
\begin{equation}\label{eq_eta*}
F(\bet^*)=\int_0^1\dP_{ex}(y_2)dy_2.
\end{equation}

Moreover, if the nozzle is expanding (contracting), which corresponds to $f'(\eta)>0 \,(<0)$ for all $\eta \in (L_0, L_1)$, it follows that
\begin{equation}
F'(\eta) = -\frac{\bK}{\bK_0} f'(\eta) \neq 0, \quad \forall \eta \in (L_0, L_1).
\end{equation}

The monotonicity of $F(\eta)$ therefore guarantees the existence of a unique solution $\bet^*$ to \eqref{eq_eta*}. Consequently, the initial approximate position $\bet^*$ is uniquely determined in this case. 

We establish the following well-posedness result for problem \eqref{BVP_dot_u1_u2} without requiring a priori information of the free boundary position $y_1=\bet^*$.
\begin{lemma}\label{lem_p_th_dot+}
Assume that \eqref{assump_bar_ka} and \eqref{compa_Pe} hold. If $f'(\eta) \neq 0$ for all $\eta \in (L_0, L_1)$, then there exists a unique solution $\bet^* \in (L_0, L_1)$ to \eqref{eq_eta*}.
Moreover, the boundary value problem \eqref{BVP_dot_u1_u2} admits a unique solution $(\du_{1+},\du_{2+})$ satisfying 
\begin{equation}\label{est_p_th_dot}
\begin{split}
\|(\du_{1+},\du_{2+})\|_{1,\al;\dN_+}^{(-\al)}&\leq C\left(\si\| f'\|_{1,\al;\dW_{1+}}
+\si\| P_{ex}\|_{1,\al;\dGa_1}^{(-\al)}+\|\du_{1-}\|_{2,\al;\dN_-}
\right)\\
&\leq C(\bar\mlU_\pm,L_0,L_1)\si,
\end{split}
\end{equation}
with $\al\in(0,1)$. The constant $C$ depends on $L_0,L_1,\bar\mlU_\pm$ and the boundary datum.
\end{lemma}
\begin{proof}
The first equation in \eqref{BVP_dot_u1_u2} shows that there exists a function $\phi(y_1,y_2)$ such that
\begin{equation}
\p_{y_1}\phi=-\brho \bu \du_{2+},\,\p_{y_2}\phi=(1-\bM_+^2)\du_{1+},\,\phi(L_0,0)=0.
\end{equation}

Then, the system \eqref{BVP_dot_u1_u2} is changed to the elliptic equation of second order in terms of $\phi$:
\begin{equation}\label{eq_phi}
-\frac{1-\brho_+\bka^2}{\brho \bu}\p_{y_1}^2\phi-\frac{\brho \bu C(\bar\mlU_+)}{1-\bM_+^2}\p_{y_2}^2\phi
=\frac{\brho_+ \bc_+^2 b_s^s}{(\ga-1)\bS_+}\p_{y_2}\du_{1-}(\bet^*,y_2),
\end{equation}
with boundary conditions:
\begin{equation}\label{bc_phi}
\begin{split}
&\p_{y_1}\phi(y_1,\i)=-\brho \bu \dg_{2\i+2}(y_1),\\
&\p_{y_2}\phi(\bet^*,y_2)=(1-\bM_+^2)\dg_1(y_2),\\
&\p_{y_2}\phi(L_1,y_2)=(1-\bM_+^2)\dg_3(y_2).
\end{split}
\end{equation}

Since both $1-\brho_+\bka^2,\,C(\bar\mlU_+)$ and $1-\bM_+^2$ are positive, the standard theory \cite{GT2001} of second-order elliptic equations implies that there exists a unique solution $\phi(y_1,y_2)$ to the boundary value problem \eqref{eq_phi}-\eqref{bc_phi} satisfying
\begin{equation}
\begin{split}
\|\phi\|_{2,\al;\dN_+}^{(-1-\al)}&\leq C\left( \|\dg_4(y_1)\|_{1,\al;\dW_{1+}}+\|\dg_1\|_{1,\al;\dGa_s}^{(-\al)}+\|\dg_3\|_{1,\al;\dGa_1}^{(-\al)}\right).
\end{split}
\end{equation}

This inequality, combined with the estimate \eqref{est_dot_U-}, immediately yields the estimate \eqref{est_p_th_dot}. The proof is complete. 
\end{proof}

Collecting all the above results, we state the main result of this section below.
\begin{theorem}\label{thm_dot_U} Under the assumptions in Lemma \ref{lem_p_th_dot+}, there exists a unique solution $(\dot\mlU_-,\dot\mlU_+;\deta',\bet^*)$ satisfying the following properties:
\begin{enumerate}

\item The position of the free surface $y_1=\bet^*$ is determined by \eqref{comp_condi}.

\item In the supersonic domain, $\dot\mlU_-=(\dot\u_-,\dS_-,0,0)$ solves the linearized MHD equations \eqref{eq_dot_U-} with the boundary conditions \eqref{bd_U-_Ga0}-\eqref{bd_U-_wall}.

\item Behind the free boundary $y_1=\bet^*$, $\dot\mlU_+=(\dot\u_+,\dS_+,0,0)$ satisfies the linear boundary value problem \eqref{BVP_dotS+} and \eqref{BVP_dot_u1_u2}. 

\item The shape of the shock front $\deta'(y_2)$ is then determined by the leading term of $\mlG_0$, that is
\begin{equation*}
\deta'(y_2)=\frac{1}{[\bp]}\left( \beta_0^+\cdot\dot\mlU_+(\bet^*,y_2)+\dg_0\right),
\end{equation*}
where $\dg_0=\beta_0^-\cdot\dot\mlU_-(\bet^*,y_2)$.

\end{enumerate}

Moreover, the solution $(\dot\mlU_\pm;\deta')$ satisfies the estimates
\begin{equation}\label{est_dotU+-_deta}
\|\dot\mlU_-\|_{2,\al;\dN_-}+\|\dot\mlU_+\|_{1,\al;\dN_+}^{(-\al)}+\|\dot\eta'\|_{1,\al;\dN_+}^{(-\al)}\leq 
C\si \left( \| f'\|_{2,\al;\dW_{1}}+\| P_{ex}\|_{2,\al;\dGa_1} \right),
\end{equation}
where $C$ depends only on $L_0,L_1$ and $\bU_\pm$.
\end{theorem}

\section{The nonlinear shock problem}\label{sec4}
Building upon the solution $(\dot\mlU_-,\dot\mlU_+;\deta',\bet^*)$ established in Theorem \ref{thm_dot_U}, we develop an iteration scheme in this section to resolve the nonlinear shock problem. Notably, the free boundary $\dGa_s$ is adopted as the initial approximation for the shock front location.

The MHD equations \eqref{defcurl_sys_y} can be rewritten as
\begin{equation}\label{MHD_A_U}
\begin{cases}
\mlA_1(\mlU)\,\p_{y_1}\u+\mlA_2(\mlU)\p_{y_2}\u=\f(\mlU),\\
\p_{y_1}S=\p_{y_1}B=\p_{y_1}\ka=0.
\end{cases}
\end{equation}

We seek shock solutions $(\mlU_\pm;\eta(y_2))$ to the MHD equation \eqref{MHD_A_U} with the following form:
\begin{equation}
\mlU_\pm=\bar\mlU_\pm+\dd\mlU_\pm,\,\dd\mlU_\pm:=(\ddu_{1\pm},\ddu_{2\pm},\ddS_\pm,\ddB_\pm,\ddka_\pm).
\end{equation}

\subsection{The supersonic solution in \texorpdfstring{$N^y$}{}}
Starting from the boundary conditions 
\begin{equation}
(S_-,B_-,\ka_-)(L_0,y_2) = (\bS_-,\bB,\bka)
\end{equation}
at the entrance, the second equation in \eqref{MHD_A_U} implies that
\begin{equation}
(S_-,B_-,\ka_-)(y_1,y_2) \equiv (\bS_-,\bB,\bka).
\end{equation}

We need only consider the boundary value problem for $\u_-$:
\begin{equation}\label{BVP_dd_U-}
\begin{cases}
\mlA_1(\mlU_-)\,\p_{y_1}\u_-+\mlA_2(\mlU_-)\p_{y_2}\u_-=0,\,\text{in}\,N^y,\\
\u_-=\bar\u_-,\,\text{on}\,\Ga_0^y,\\
u_{2-}-\i \si f' u_{1-}=0,\,\text{on}\,W_{\i}^y.
\end{cases}
\end{equation}

The boundary value problem \eqref{BVP_dd_U-} constitutes a hyperbolic system in the supersonic regime. Consequently, applying the established theory for quasi-linear hyperbolic systems yields the following existence and uniqueness results.

\begin{theorem}
Under the assumption that condition \eqref{comp_nozzle} holds, there exists a positive constant $\si_1<1$ depending on the background solution $\bar\mlU_-$ and the length of the nozzle such that for any $0<\si<\si_1$, the initial boundary value problem \eqref{BVP_dd_U-} admits a unique solution satisfying the estimate
\begin{equation}\label{est_dd_bar_U-}
\|\u_- - \bar\u_-\|_{2,\al;\overline{N^y}}\leq C(\bar\mlU_-,L_0,L_1)\si.
\end{equation}
Furthermore, when compared with the supersonic solution $\dot\mlU_-$ constructed in Lemma \ref{lem_dot_U-}, there holds
\begin{equation}\label{est_ddU_du-}
\|\dd\u_- - \dot\u_-\|_{1,\al;\overline{N^y}}\leq C(\bar\mlU_-,L_0,L_1)\si^2.
\end{equation}
\end{theorem}

\begin{proof}
The unique existence of $\u_-\in C^{2,\al}(\overline{N^y})$ is guaranteed by the standard theory of quasi-linear hyperbolic systems, as established in \cite{LiYU}. A direct computation derives

\begin{equation*}
\begin{cases}
\mlA_1(\bar\mlU_-)\,\p_{y_1}(\ddot\u_- - \dot\u_-) + \mlA_2(\bar\mlU_-)\,\p_{y_2}(\ddot\u_- - \dot\u_-)=\mlF_-(\dd\mlU_-),\,\text{in}\,N^y,\\
\dd\u_- - \dot\u_-=0,\,\text{on}\,\Ga_0^y,\\
(\ddu_{2-}-\du_{2-})-\i \si f' (\ddu_{1-}-\du_{1-})=\i \si f' \du_{1-},\,\text{on}\,W_{\i}^y.
\end{cases}
\end{equation*}
where
\begin{equation*}
\mlF_-(\dd\mlU_-)=\sum_{i=1}^2(\mlA_i(\bar\mlU_-)-\mlA_i(\mlU_-))\p_{y_i}\dd\u_-.
\end{equation*}

It follows that
\begin{equation}
\begin{split}
\|\dd\u_- - \dot\u_-\|_{1,\al;\overline{N^y}}&\leq C \lrb{\|\mlF_-(\dd\mlU_-)\|_{1,\al;\overline{N^y}}+ \|\si f' \du_{1-}\|_{1,\al;W_1^y} }\\
&\leq C \lrb{ \|\dd\mlU_-\|_{1,\al;\overline{N^y}}\|\dd\u_-\|_{2,\al;\overline{N^y}}+ \si\| \du_{1-}\|_{1,\al;W_1^y} }\\
&\leq  C(\bar\mlU_-,L_0,L_1)\si^2.
\end{split}
\end{equation}

Combining this with \eqref{est_dd_bar_U-}, we obtain the desired estimate \eqref{est_ddU_du-}, which completes the proof.
\end{proof}

\subsection{The shock front and subsonic solution}
The transonic shock front $\Ga_s^y$ is represented as
\begin{equation}
\Ga_s^y=\{(y_1,y_2):y_1=\eta(y_2)=\bet^*+\ddeta(y_2),\,0<y_2<1\}.
\end{equation}

In the subsonic domain, $(\mlU_+,\eta(y_2))$ solves the following problem:
\begin{equation}\label{BVP_dd_U+_y}
\begin{cases}
\mlA_1(\mlU_+)\,\p_{y_1}\u_+ +\mlA_2(\mlU_+)\p_{y_2}\u_+  =\f(\mlU_+),\,\text{in}\,N_+^y,\\
\p_{y_1}S_+=\p_{y_1}B_+=\p_{y_1}\ka_+=0,\,\text{in}\,N_+^y,\\
[B]=[\ka]=0,\,\text{on}\,\Ga_s^y,\\
G_i(\mlU_+,\mlU_-)=0,\,i=0,1,2,\,\text{on}\,\Ga_s^y,\\
\mlP(\mlU_+)=\bP_+ +\si \tP_{ex}(y_2),\,\text{on}\,\Ga_{1}^y,\\
u_{2+}-\i \si f' u_{1+}=0,\,\text{on}\,W_{\i}^y.
\end{cases}
\end{equation}

The second and third equations in \eqref{BVP_dd_U+_y} immediately yield
\begin{equation}\label{exp_Bka+}
\begin{split}
(B_+,\ka_+)(y_1,y_2)&=(B_+,\ka_+)(\eta(y_2),y_2)=(B_-,\ka_-)(\eta(y_2),y_2)\\
&\equiv(\bB,\bka).
\end{split}
\end{equation}

Then, the system \eqref{BVP_dd_U+_y} is a nonlinear free boundary value problem for $(u_{1+},u_{2_+},S_+;\eta)$, with the shock front $y_1=\eta(y_2)$ becomes part of the solution. To fix the shock profile, we employ a coordinate transformation that straightens the curved shock profile, thereby converting the free boundary value problem into a fixed boundary formulation. 

Specifically, we introduce the following transformation:
\begin{equation}\label{chag-var-fix-bc}
\begin{cases}
z_1=\frac{L_1-\bet^*}{L_1-\eta(y_2)}(y_1-\eta(y_2))+\bet^*,\\
z_2=y_2.
\end{cases}
\end{equation}
And the inverse change  variable gives
\begin{equation}
\begin{cases}
y_1:=Y_1(z_1,z_2;\eta)=z_1+\frac{L_1-z_1}{L_1-\bet^*}(\eta(z_2)-\bet^*),\\
y_2=z_2.
\end{cases}
\end{equation}

Under this transformation, the shock profile $\Ga_s^y$ and the subsonic region $N_+^y$ are mapped to the fixed regions:
\begin{equation}
\begin{split}
\Ga_s^z&=\{(z_1,z_2):z_1=\bet^*,0<z_2<1\},\\\
N_+^z&=\{(z_1,z_2):\bet^*<z_1<L_1,0<z_2<1\},
\end{split}
\end{equation}
with the nozzle wall $W_{\i+}^z=\dW_{\i+}$ and exit $\Ga_1^z=\dGa_1$ remaining invariant.

The unknown function, under transformation \eqref{chag-var-fix-bc}, is expressed in $z$-coordinates as
\begin{equation*}
\begin{split}
\hat\mlU_+:&=(\hu_{1+},\hu_{2+},\hS_+,\hB_+,\hka_+)(z_1,z_2)\\
&=(\tu_{1+},\tu_{2+},\tS_+,\tB_+,\tka_+)(Y_1(z_1,z_2;\eta),z_2),
\end{split}
\end{equation*}

A simple calculation shows that
\begin{equation}\no
\begin{aligned}
\p_{y_1}&=\frac{L_1-\bet^*}{L_1-\eta(z_2)}\p_{z_1},\\
\p_{y_2}&=-\frac{L_1-z_1}{L_1-\eta(z_2)}\eta'(z_2)\p_{z_1}+\p_{z_2}.
\end{aligned}
\end{equation}


Then, the free boundary value problem for $(u_{1+},u_{2_+},S_+)(y_1,y_2)$ is changed to
\begin{equation}\label{BVP_ddU+_z}
\begin{cases}
\mlA_1(\hat\mlU_+)\,\p_{z_1}\hat\u_+ +\mlA_2(\hat\mlU_+)\p_{z_2}\hat\u_+  =\mlF(\hat\mlU_+,\eta),\,\text{in}\,N_+^z,\\
\p_{z_1}\hS_+=0,\,\text{in}\,N_+^z,\\
G_i(\hat\mlU_+,\mlU_-(\eta(z_2),z_2))=0,\,i=1,2,\,\text{on}\,\Ga_s^z,\\
\mlP(\hat\mlU_+)=\bP_+ +\si \tP_{ex}(z_2),\,\text{on}\,\Ga_{1}^z,\\
\hu_{2+}-\i \si f'(Y_2(z_1,\i;\eta(\i))) \hu_{1+}=0,\,\text{on}\,W_{\i}^z.
\end{cases}
\end{equation}
where the source term $\mlF$ is given by 
\begin{equation}
\begin{split}
\mlF(\hat\mlU_+,\eta)&=\f(\hat\mlU_+)
-\frac{\eta(z_2)-\bet^*}{L_1-\eta(z_2)}\mlA_1(\hat\mlU_+)\p_{z_1}\hat\mlU_+\\
&+\frac{L_1-z_1}{L_1-\eta(z_2)}\eta'(z_2)\mlA_2(\hat\mlU_+)\p_{z_1}\hat\mlU_+.
\end{split}
\end{equation}

For notational simplicity, we will omit the hat notation $\hat{}$ and the subscript $+$ in $\hat\mlU_+$ throughout the remainder of this section.

\subsection{The linearized system and iteration scheme} We try to address the nonlinear boundary value problem \eqref{BVP_ddU+_z} by the iteration method. The shock front $y_1=\eta(z_2)$ is fully characterized through two components: the shape of the shock profile $\eta'$, derived from the Rankine-Hugoniot conditions, and the terminal value $\eta^*:=\eta(1)$, obtained by fulfilling the compatibility condition. To our purpose, we decompose the shock front as
\begin{equation}\label{sharp_shock}
\eta(z_2)=\bet^*+\dd\eta^*-\int_{z_2}^{1}\ddeta'(\tau)\,d\tau.
\end{equation}
where
\begin{equation*}
\ddeta^*:=\ddeta(1)=\eta^*-\bet^*.
\end{equation*}

It follows from \eqref{exp_Bka+} that
\begin{equation}\label{exp_sharp_B_ka}
(B,\ka)(z_1,z_2)\equiv(\bB,\bka),
\end{equation}
which implies that the fluctuation in the solution vector must be of the form $\dd\mlU=(\ddu_1,\ddu_2,\ddS,0,0)$. Therefore, the iteration process essentially involves the first three components $(\ddu_{1},\ddu_2,\ddS)(z_1,z_2)$.


Given $(\dd\mlU,\ddeta')$ which is close to the initial approximation $(\dot\mlU,\dot\eta')$ constructed in Theorem \ref{thm_dot_U} and the upstream solution obtained in Lemma \ref{lem_dot_U-}, we aim to determine $\dd\eta^*$ and $(\dd\mlU_\sharp;\dd\eta_\sharp')$ by solving the corresponding linear system.

We begin by prescribing the boundary conditions on the fixed shock front $z_1=\bet^*$ as
\begin{equation*}
\bar\mlB_{s+} \dd\mlV_\sharp(\bet^*,z_2)
={\mf{g}}(\dd\mlU(\bet^*,z_2),\dd\mlU_-(\eta(z_2),z_2),\dd\eta'(z_2);\ddeta^*),
\end{equation*}
where $\dd\mlV_\sharp=(\ddu_{\sharp 1},\ddS_\sharp)^\top$ and the vector ${\bf{g}}(z_2)=(g_1(z_2),g_2(z_2))^\top$ is defined by
\begin{equation}\label{def_mfg}
{\mf{g}}(\dd\mlU,\dd\mlU_-,\dd\eta';\ddeta^*)=
\bar\mlB_{s+}\dd\mlV-\mf{G}(\mlU,\mlU_-(\eta(z_2),z_2))
\end{equation}
with $\dd\mlV=(\ddu_1,\ddS)$ and $\mf{G}=(G_1,G_2)$.

Noting that $\bar\mlB_{s+}$ is invertible, the boundary conditions on $\Ga_s^z$ are
\begin{equation}\label{bdy_u1_S_sharp}
\begin{split}
\ddS_\sharp(\bet^*,z_2)&=\ddg_s(z_2;\dd\mlU,\dd\mlU_-,\dd\eta';\ddeta^*),\\
\ddu_{\sharp,1}(\bet^*,z_2)&=\ddg_1(z_2;\dd\mlU,\dd\mlU_-,\dd\eta';\ddeta^*).
\end{split}
\end{equation}
where 
\begin{equation}\label{def_ddgs_ddg1}
(\ddg_s,\ddg_1)^\top=\bar\mlB_{s+}^{-1} \mf{g}.
\end{equation}

For given  $(\dd\mlU;\dd\eta^*)$, we get an update $(\ddu_{\sharp,1},\ddu_{\sharp,2},\ddS_\sharp)$ by solving the following linearized system:
\begin{equation}\label{eq_sharp}
\begin{cases}
(1-\bM_+^2)\p_{z_1}\ddu_{\sharp 1} +\brho \bu \p_{z_2} \ddu_{\sharp 2}=\dd{f}_1(\dd\mlU,\dd\eta';\dd\eta^*),\\
(1-\brho_+ \bka^2)\p_{z_1}\ddu_{\sharp 2}-\brho \bu C(\bar\mlU_+) \p_{z_2}\ddu_{\sharp 1}=\dd{f}_2(\dd\mlU,\dd\eta';\dd\eta^*),\\
\p_{z_1}\ddS_\sharp=0.
\end{cases}
\end{equation}

The entropy $\ddS_\sharp$ is completely determine by its value on $\Ga_s^z$, i.e.,
\begin{equation*}
\ddS_\sharp(z_1,z_2)=\ddg_s(z_2;\dd\mlU,\dd\mlU_-,\dd\eta';\ddeta^*).
\end{equation*}

The high order terms $\dd{f}_1,\dd{f}_2$ are given by
\begin{equation}\label{def_f1}
\begin{split}
\dd{f}_1(\dd\mlU,\dd\eta';\dd\eta^*)=&(M_1^2-\bM_+^2)\p_{z_1}\ddu_{1} +(\rho u_1 -\brho \bu) \p_{z_2} \ddu_{ 2}\\
&+M_1M_2\p_{z_1}u_2+\rho u_2 \p_{z_2}u_1\\
&-\frac{\ddeta^*-\int_{z_2}^{1}\ddeta'(\tau)\,d\tau}{L_1-\eta(z_2)}
\lrb{(1-M_1^2)\p_{z_1}\ddu_{1} -M_1 M_2 \p_{z_1}\ddu_{2}}\\
&+\frac{L_1-z_1}{L_1-\eta(z_2)}\dd\eta' \lrb{-\rho u_2\p_{z_1}\ddu_{1} +\rho u_1 \p_{z_1}\ddu_{2}},
\end{split}
\end{equation}
and
\begin{equation}\label{def_f2}
\begin{split}
\dd{f}_2(\dd\mlU,\dd\eta';\dd\eta^*)=&(\rho \ka^2-\brho_+ \bka^2)\p_{z_1}\ddu_{ 2}-(\brho \bu C(\bar\mlU_+) - \rho u_1C(\mlU))\p_{z_2}\ddu_{ 1}\\
&-\rho \ka^2 M_1 M_2\p_{z_1}u_1-\rho \ka^2 M_2^2\p_{z_1}u_2+\rho u_2 C(\mlU)\p_{z_2}u_2\\
&-\frac{\ddeta^*-\int_{z_2}^{1}\ddeta'(\tau)\,d\tau}{L_1-\eta(z_2)}
\lrb{\rho \ka^2 M_1 M_2\p_{z_1}\ddu_{1} +(C(\mlU)-\rho\ka^2 M_1^2) \p_{z_1}\ddu_{2}}\\
&-\frac{L_1-z_1}{L_1-\eta(z_2)}\dd\eta' C(\mlU) \lrb{\rho u_1\p_{z_1}\ddu_{1} +\rho u_2 \p_{z_1}\ddu_{2}}\\
&+\rho^\ga \frac{1+\ga \rho \ka^2 |{\bf M}|^2}{\ga-1}\p_{z_2}\ddg_s(z_2;\dd\mlU,\dd\mlU_-,\dd\eta';\ddeta^*).
\end{split}
\end{equation}

The boundary conditions on the  walls $W_{\i\pm}^z$ are given by
\begin{equation}\no
\ddu_{\sharp,2}=\i \si f'\lrb{z_1+\frac{L_1-z_1}{L_1-\bet^*}\ddeta^*} (\bu_++\ddu_{1}):=\ddg_{2\i+2}(z_1;\dd\mlU;\ddeta^*),\,\,\text{on}\,W_{\i+}^z. 
\end{equation}

Restricting \eqref{def_mlP} on $z_1=L_1$, we obtain 
\begin{equation}
\si P_{ex}(X_2(L_1,z_2;\dd\mlU))=-\brho \bu C(\bar\mlU_+)\ddu_{\sharp,1}-\frac{\bp_+ + \bka^2\brho^2 \bu^2}{(\ga-1)\bS_+}\ddS_\sharp+R_p(z_2;\dd\mlU),
\end{equation}
where
\begin{equation*}
X_2(L_1,z_2;\dd\mlU)=\int_{0}^{z_2}\frac{1}{(\rho u_1)(L_1,\tau)}d\tau,
\end{equation*}
and $R_p(z_2;\dd\mlU)$ is an error of second order defined as
\begin{equation*}
R_p(z_2;\dd\mlU)=\mlP(\mlU)-\bP_+ + \brho \bu C(\bar\mlU_+)\ddu_{1}+\frac{\bp_+ + \bka^2\brho^2 \bu^2}{(\ga-1)\bS_+}\ddS.
\end{equation*}

The boundary condition at the exit is thus prescribed as 
\begin{equation}
\ddu_{\sharp,1}(L_1,z_2)=\ddg_3(z_2;\dd\mlU,\dd\mlU_-,\dd\eta';\ddeta^*) ,
\end{equation}
where
\begin{equation}\label{def_ddg3}
\begin{split}
\ddg_3(z_2;\dd\mlU,\dd\mlU_-,\dd\eta';\ddeta^*)=&-\si \frac{P_{ex}(X_2(L_1,z_2;\dd\mlU))}{\brho\bu C(\bar\mlU_+)}\\
&-\frac{\bp_+ + \bka^2\brho^2 \bu^2}{(\ga-1)\brho\bu C(\bar\mlU_+)\bS_+}\ddg_s(z_2)+\frac{R_p(z_2;\dd\mlU)}{\brho\bu C(\bar\mlU_+)}.
\end{split}
\end{equation}

To summarize, the entropy $\ddS_\sharp$ is directly given by 
\begin{equation}\label{exp_ddS_sharp}
\ddS_\sharp(z_1,z_2)=\ddg_s(z_2;\dd\mlU,\dd\mlU_-,\dd\eta';\ddeta^*),
\end{equation}
while the velocity field $\dd\u_\sharp=(\ddu_{\sharp,1},\ddu_{\sharp,2})$ will be determined by solving the following boundary value problem:
\begin{equation}\label{BVP_dd_sharp_u1u2}
\begin{cases}
\mlA_1(\bar\mlU_+)\p_{z_1}\dd\u_\sharp +\mlA_2(\bar\mlU_+)\p_{z_2} \dd\u_\sharp=\dd\mlF(\dd\mlU,\dd\eta';\dd\eta^*),\,\text{in }N_+^z,\\
\ddu_{\sharp,1}(\bet^*,z_2)=\ddg_1(z_2;\dd\mlU,\dd\mlU_-,\dd\eta';\ddeta^*),\,\text{on }\Ga_s^z,\\
\ddu_{\sharp,1}(L_1,z_2)=\ddg_3(z_2;\dd\mlU,\dd\mlU_-,\dd\eta';\ddeta^*),\,\text{on }\Ga_1^z,\\
\ddu_{\sharp,2}(z_1,\i)=\ddg_{2\i+2}(z_1;\dd\mlU;\ddeta^*),\,\text{on }W_{\i+}^z,
\end{cases}
\end{equation}
and the shape of the shock front $\ddeta'_\sharp$ is determined by 
\begin{equation}\label{eq_sharp_dd_eta}
\ddeta'_\sharp=\frac{1}{[\bp]}\lrb{\beta_0^+\cdot \dd\mlU_\sharp+\ddg_0(z_2;\dd\mlU,\dd\mlU_-,\dd\eta';\ddeta^*)},
\end{equation}
where $\dd\mlF=(\dd{f}_1,\dd{f}_2)$, and
\begin{equation}
\ddg_0(z_2;\dd\mlU,\dd\mlU_-,\dd\eta';\ddeta^*)= \ddeta'[\bp]-\beta_0^+\cdot \dd\mlU+G_0(\mlU,\mlU_-(\eta(z_2),z_2)).
\end{equation}

\subsection{The solvability condition and well-posedness for the iteration system} Following the same methodology employed in constructing $(\dot\mlU_+,\deta';\bet^*)$ in the previous section, we can determine $\ddeta^*$ through the solvability condition. The  solvability condition to the the first-order elliptic system $(\ddu_{\sharp,1},\ddu_{\sharp,2})$ in \eqref{BVP_dd_sharp_u1u2} can be formulated as
\begin{equation}\label{nonlin_com_condi}
\I(\dd\eta^*;\dd\mlU,\dd\eta',\dd\mlU_-)=0,
\end{equation}
where
\begin{equation*}
\begin{split}
\I(\dd\eta^*;\dd\mlU,\dd\eta',\dd\mlU_-)=&
\int_{\bet^*}^{L_1}\ddg_4(z_1;\dd\mlU;\dd\eta^*)\,dz_1\\
&-\frac{1-\bM_+^2}{\brho \bu}\int_0^1\lrb{\ddg_1 -\ddg_3}(z_2;\dd\mlU,\dd\mlU_-,\dd\eta';\ddeta^*)\,dz_2\\
&-\frac{1}{\brho \bu}\int_{N_+^z}\dd{f}_1(\dd\mlU,\dd\eta';\dd\eta^*)\,dz_1dz_2\\
:=&\I^1+\I^2+\I^3.
\end{split}
\end{equation*}

For any $\vep>0$, we define the function space
\begin{equation}\label{def_spa_dN}
\mr{\dN}_{\vep}:=\left\{(\dd\mlU,\ddeta'):
\|\dd\mlU-\dot\mlU_+\|_{1,\al;N_+^z}^{(-\al)}+\|\ddeta'-\dot\eta'\|_{1,\al;\Ga_s^z}^{(-\al)}\leq \vep
\right\},
\end{equation}
which constitutes an $\vep$-neighborhood of $(\dot\mlU_+,\deta')$.

The following lemma can be established.
\begin{lemma}\label{lem_nonlin_comp}
There exists a small constant $0<\si_2<\si_1$ such that for any $\si\in(0,\si_2)$, if $(\dd\mlU,\ddeta')\in\mr{\dN}_{\si^{3/2}}$ and $f'(\bet^*)\neq 0$, then the equation \eqref{nonlin_com_condi} admits a unique solution $\ddeta^*$ satisfying
\begin{equation}\label{est_ddeta*}
|\ddeta^*|\leq C\si,
\end{equation}
where the constant $C$ depends on the background flow.
\end{lemma}
\begin{proof}
It can be verified that
\begin{equation}\label{I_od0}
\I(0;0,0,\dot\mlU_-)=0.
\end{equation}
Indeed, \eqref{I_od0} coincides with the equation \eqref{comp_condi}.

To derive an expansion of  $\I(\dd\eta^*;\dd\mlU,\dd\eta',\dd\mlU_-)$ in a neighborhood of $(0;0,0,\dot\mlU_-)$, we estimate each term $\I^j,j=1,2,3$, individually. Before proceeding, we first note that for any $(\dd\mlU,\ddeta')\in\mr{\dN}_{\si^{3/2}}$, the estimate
\begin{equation}\label{assp_est_ddU}
\|\dd\mlU\|_{1,\al;N_+^z}^{(-\al)}+\|\dd\eta'\|_{1,\al;\Ga_s^z}^{(-\al)}\leq C\si,
\end{equation}
can be readily verified using \eqref{est_dotU+-_deta}.

{\bf Estimation of $\I^1$.} A change of variable yields
\begin{equation}\label{est_I1}
\begin{split}
\I^1=&\si\int_{\bet^*}^{L_1}f'\lrb{z_1+\frac{L_1-z_1}{L_1-\bet^*}\ddeta^*} (\bu_++\ddu_{1}(z_1,1))dz_1\\
=&\si \bu_+\int_{\bet^*+\dd\eta^*}^{L_1}f'(\tau)d\tau+ \si\frac{\dd\eta^*}{L_1-\eta^*} \bu_+\int_{\bet^*+\dd\eta^*}^{L_1}f'(\tau)d\tau\\
&+O(1)\si^2(1+\ddeta^*).
\end{split}
\end{equation}

{\bf Estimation of $\I^2$.} By \eqref{def_ddg3}, one has
\begin{equation*}
\begin{split}
\I^2=&-\frac{1-\bM_+^2}{\brho \bu}\int_0^1 \ddg_1(z_2;\dd\mlU,\dd\mlU_-,\dd\eta';\ddeta^*)\,dz_2\\
&-\frac{1-\bM_+^2}{\brho \bu}\int_0^1 \frac{\bp_+ + \bka^2\brho^2 \bu^2}{(\ga-1)\brho\bu C(\bar\mlU_+)\bS_+}\ddg_s(z_2;\dd\mlU,\dd\mlU_-,\dd\eta';\ddeta^*) \,dz_2\\
&-\frac{1-\bM_+^2}{\brho \bu}\si\int_0^1  \frac{P_{ex}(X_2(L_1,z_2;\dd\mlU))}{\brho\bu C(\bar\mlU_+)}\,dz_2\\
&+\frac{1-\bM_+^2}{\brho \bu}\int_0^1 \frac{R_p(z_2;\dd\mlU)}{\brho\bu C(\bar\mlU_+)} \,dz_2\\
:=&\sum_{j=1}^{4}\I^2_j
\end{split}
\end{equation*}

It's clear that
\begin{equation}\label{est_I2_4}
\I_4^2=O(1)\si^2.
\end{equation}

Recalling \eqref{def_mfg}, the definition of $\mf{g}$, we have
\begin{equation}
\begin{split}
{\bf g}=&\bar\mlB_{s+}\dd\mlV-\mlG(\mlU,\mlU_-(\eta(z_2),z_2))\\
=&\bar\mlB_{s+}\dd\mlV-\bar\mlB_{s-}\dd\mlV_-(\eta(z_2),z_2)-\mlG(\mlU,\mlU_-(\eta(z_2),z_2))\\
&+\bar\mlB_{s-}\left(\dd\mlV_-(\eta(z_2),z_2)-\dot\mlV_-(\eta(z_2),z_2)\right)\\
&+\bar\mlB_{s-}\left(\dot\mlV_-(\eta(z_2),z_2)-\dot\mlV_-(\bet^*+\ddeta^*,z_2)\right)\\
&+\bar\mlB_{s-}\dot\mlV_-(\bet^*+\ddeta^*,z_2)\\
:=&\sum_{i=1}^{4}{\bf g}^i.
\end{split}
\end{equation}

It's obvious that
\begin{equation}\label{est_g1}
{\bf g}^1=O(1)\si^2.
\end{equation}

The inequality \eqref{est_ddU_du-} implies
\begin{equation}\label{est_g2}
{\bf g}^2=\bar\mlB_{s-}\left( \dd\mlV_-(\eta(z_2),z_2)-\dot\mlV_-(\eta(z_2),z_2)\right)= O(1)\si^2.
\end{equation}

By using \eqref{est_dot_U-} and \eqref{assp_est_ddU}, we obtain
\begin{equation}\label{est_g3}
\begin{split}
{\bf g}^3&=
\bar\mlB_{s-}\left(\dot\mlV_-(\eta(z_2),z_2)-\dot\mlV_-(\bet^*+\ddeta^*,z_2)\right)\\
&=-\bar\mlB_{s-} \int_{z_2}^{1}\dd\eta'(\tau)\,d\tau\, \int_{0}^{1}\p_{z_1}\dot\mlV_-(s\eta(z_2)+(1-s)(\bet^*+\ddeta^*),z_2)\,ds\\
&=O(1)\si^2.
\end{split}
\end{equation}

Therefore, by \eqref{est_g1}-\eqref{est_g3}, we derive
\begin{equation}\label{est_ddgs_ddg1}
\begin{split}
(\ddg_s,\ddg_1)&=\bar\mlB_{s+}^{-1} {\bf g}= \bar\mlB_{s+}^{-1}\bar\mlB_{s-}\dot\mlV_-(\bet^*+\ddeta^*,z_2)+O(1)\si^2\\
&=(b_u^s,b_s^s)^\top \du_{1-}(\bet^*+\ddeta^*,z_2)+O(1)\si^2.
\end{split}
\end{equation}

As a result, we conclude that
\begin{equation}\label{est_I2_12}
\begin{split}
\I_2^1+\I_2^2=&-\frac{1-\bM_+^2}{\brho \bu}b_u^s\int_0^1  \du_{1-}(\bet^*+\ddeta^*)\,dz_2\\
&+\frac{1-\bM_+^2}{\brho \bu} b_u^1  \int_0^1 \du_{1-}(\bet^*+\ddeta^*) \,dz_2+O(1)\si^2\\
=&O(1)\si^2+\bu_+\frac{\bp_-}{\bp_+} \si f(\bet^*+\ddeta^*)\\
&-(1-\bM_+^2)\frac{(\bp_+ +\bka^2\brho^2\bu^2)}{\brho \bu C(\bar\mlU_+)} \frac{[\bp]}{\bp_+} \si f(\bet^*+\ddeta^*).
\end{split}
\end{equation}

Since $\brho\bu=1$, we have
\begin{equation*}
\begin{split}
P_{ex}(X_2(L_1,z_2;\dd\mlU))&=P_{ex}\lrb{\int_{0}^{z_2}\frac{1}{(\rho u_1)(L_1,\tau)}\,d\tau}\\
&=P_{ex}(z_2)+P_{ex}\left(\int_{0}^{z_2}\frac{1}{(\rho u_1)(L_1,\tau)}\,d\tau\right)-P_{ex}(z_2)\\
&= P_{ex}(z_2)+O(1)\si.
\end{split}
\end{equation*}

It follows that
\begin{equation}\label{est_I2_3}
\begin{split}
\I_3^2&=-\si \bK_0\int_0^1  P_{ex}(X_2(L_1,z_2;\dd\mlU))\,dz_2\\
&=-\si \bK_0\int_0^1  P_{ex}(z_2)\,dz_2+O(1)\si^2.
\end{split}
\end{equation}

Combining \eqref{est_I1}-\eqref{est_I2_4} and \eqref{est_I2_12}-\eqref{est_I2_3} yields
\begin{equation}\label{est_I1_I2}
\begin{split}
\I^1+\I^2=&O(1)\si^2(1+\ddeta^*)+ \si\frac{\dd\eta^*}{L_1-\eta^*} \bu_+\int_{\bet^*+\dd\eta^*}^{L_1}f'(\tau)d\tau\\
&-\si\bK f(\bet^*+\ddeta^*)+\si\bu_+ f(L_1)-\si \bK_0\int_0^1  P_{ex}(z_2)\,dz_2\\
=&O(1)\si^2(1+\ddeta^*)+ \si\frac{\dd\eta^*}{L_1-\eta^*} \bu_+\int_{\bet^*+\dd\eta^*}^{L_1}f'(\tau)d\tau\\
&-\si \bK f'(\bet^*)\ddeta^*+O(1)\si |\ddeta^*|^2.
\end{split}
\end{equation}

{\bf Estimation of $\I^3.$} Rewrite $\dd{f}_1(\dd\mlU,\dd\eta';\dd\eta^*)$ as
\begin{equation}\label{expan_ddf2}
\begin{split}
\dd{f}_1(\dd\mlU,\dd\eta';\dd\eta^*)=&(M_1^2-\bM_+^2)\p_{z_1}\ddu_{1} +(\rho u_1 -\brho \bu) \p_{z_2} \ddu_{ 2}
+M_1M_2\p_{z_1}u_2+\rho u_2 \p_{z_2}u_1\\
&+\lrb{\frac{(1-M_1^2)\int_{z_2}^{1}\ddeta'(\tau)\,d\tau}{L_1-\eta(z_2)}-\frac{(L_1-z_1)\dd\eta' \rho u_2}{L_1-\eta(z_2)} }\p_{z_1}\ddu_1\\
&+\lrb{\frac{(L_1-z_1)\dd\eta' \rho u_1}{L_1-\eta(z_2)}
+\frac{ \ddeta^*- \int_{z_2}^{1}\ddeta'(\tau)\,d\tau}{L_1-\eta(z_2)}M_1 M_2 }\p_{z_1}\ddu_2\\
&-\ddeta^*\frac{1-M_1^2}{L_1-\eta(z_2)}\p_{z_1}\ddu_{1}\\
:=&\sum_{j=1}^4 \dd{f}_1^j.
\end{split}
\end{equation}

The first three items of $\dd{f}_1(\dd\mlU,\dd\eta';\dd\eta^*)$ can be straightforwardly estimated as follows:
\begin{equation}\label{est_f2_123}
\int_{N_+^z}\dd{f}_1^1+\dd{f}_1^2+\dd{f}_1^3\,dz_1dz_2 = O(1)\si^2\,(1+\ddeta^*).
\end{equation}

And the last term $\dd{f}_1^4$ can be further decomposed as
\begin{equation*}
\begin{split}
\dd{f}_1^4=&-\ddeta^*\frac{\bM_+^2-M_1^2}{L_1-\eta(z_2)}\p_{z_1}\ddu_{1}
-\ddeta^*\frac{1-\bM_+^2}{L_1-\eta(z_2)} \p_{z_1}\lrb{\ddu_{1}-\du_{1+} }\\
&-\ddeta^*\frac{1-\bM_+^2}{L_1-\eta(z_2)} \p_{z_1}\du_{1+}
+\frac{\ddeta^* \int_{z_2}^{1}\ddeta'(\tau) d\tau}{(L_1-\eta(z_2))(L_1-\bet^*)}\du_{1+}\\
&-\ddeta^*\frac{(1-\bM_+^2)}{L_1-\eta^*}\p_{z_1}\du_{1+}.
\end{split}
\end{equation*}


Combining the first equation in \eqref{eq_dot_U+} with the boundary condition for $\du_{2_+}$ given in \eqref{bd_U+_wall}, we obtain
\begin{equation*}
\begin{split}
-(1-\bM_+^2)\int_{N_+^z}\p_{z_1}\du_{1+}\,dz_1dz_2&= \brho \bu \int_{N_+^z}\p_{z_2}\du_{2+}\,dz_1dz_2\\
&=\brho_+ \bu_+^2 \si\int_{\bet^*}^{L_1}f'(z_1)\,dz_1.
\end{split}
\end{equation*}

It follows that
\begin{equation}\label{est_f2_4}
\begin{split}
\int_{N_+^z}\dd{f}_1^4dz_1dz_2 =& O(1)\si^2\dd\eta^*+O(1)\si^\frac{3}{2}\ddeta^*\\
&+\si \brho_+ \bu_+^2 \frac{\ddeta^*}{L_1-\eta^*}\int_{\bet^*}^{L_1}f'(\tau)d\tau.
\end{split}
\end{equation}

We finally derive from \eqref{est_f2_123} and \eqref{est_f2_4} that
\begin{equation}\label{est_I3}
\begin{split}
\I^3&=-\frac{1}{\brho \bu}\int_{N_+^z} \dd{f}_1(\dd\mlU,\dd\eta';\dd\eta^*) dz_1dz_2\\
&= O(1)\si^2+ O(1)\si^\frac{3}{2}\ddeta^*-\si \frac{\ddeta^*}{L_1-\eta^*}\bu_+\int_{\bet^*}^{L_1}f'(\tau)d\tau.
\end{split}
\end{equation}

Collecting \eqref{est_I1_I2} and \eqref{est_I3}, we obtain
\begin{equation*}
\begin{split}
&\I(\dd\eta^*;\dd\mlU,\dd\eta',\dd\mlU_-)\\
=&\si\frac{\dd\eta^*}{L_1-\eta^*} \bu_+(f(L_1)-f(\bet^*+\ddeta^*))-\si \frac{\ddeta^*}{L_1-\eta^*}\bu_+(f(L_1)-f(\bet^*))\\
&+O(1)\si^2(1+\ddeta^*)-\si \bK f'(\bet^*)\ddeta^*+O(1)\si |\ddeta^*|^2 + O(1)\si^\frac{3}{2}\ddeta^*\\
=&-\si\frac{\bu_+ }{L_1-\eta^*} (f(\bet^*+\ddeta^*)-f(\bet^*))\dd\eta^*\\
&+O(1)\si^2(1+\ddeta^*)-\si \bK f'(\bet^*)\ddeta^*+O(1)\si |\ddeta^*|^2 + O(1)\si^\frac{3}{2}\ddeta^*.
\end{split}
\end{equation*}

Consequently, we derive the expansion of $\I(\dd\eta^*;\dd\mlU,\dd\eta',\dd\mlU_-)$:
\begin{equation}\label{expan_I}
\I(\dd\eta^*;\dd\mlU,\dd\eta',\dd\mlU_-)=
-\si \bK f'(\bet^*)\ddeta^*+ O(1)\si |\dd\eta^*|^2+ O(1)\si^\frac{3}{2}\ddeta^*+O(1)\si^2.
\end{equation}

The above expansion yields
\begin{equation}
\frac{\p\I}{\p\ddeta^*}(0;0,0,\dot\mlU_-)=-\si \bK f'(\bet^*)+ O(1)\si^\frac{3}{2}\neq 0,
\end{equation}
provided that $f'(\bet^*)\neq 0$ and $\si > 0$ is chosen sufficiently small.

An application of the implicit function theorem guarantees the unique existence of the solution $\ddeta^*$ to the equation \eqref{nonlin_com_condi} for given $\dd\mlU_-$ and $(\dd\mlU,\dd\eta')\in\mr{\dN}_{\si^{3/2}}$. The estimate \eqref{est_ddeta*} is then an immediate consequence of the expansion \eqref{expan_I}. The proof is complete.
\end{proof}

Under the validity of the compatibility condition, it's easy to establish the following well-posedness lemma for the first-order elliptic system \eqref{eq_sharp}.

\begin{lemma}
Under the same assumptions as in Lemma \ref{lem_nonlin_comp}, let $\ddeta^*$ be the unique solution to equation \eqref{nonlin_com_condi}. Then, the boundary value problem \eqref{BVP_dd_sharp_u1u2} admits a unique solution $\dd\u_\sharp$ satisfying the estimate
\begin{equation}\label{est_sharp_dd_u1_u2}
\begin{split}
\|\dd\u_\sharp\|_{1,\al;N_+^z}^{(-\al)}\leq& C\bigg(
\|\dd\mlF\|_{0,\al;N_+^z}^{(-\al+1)}
+\|\ddg_1\|_{1,\al;\Ga_s^z}^{(-\al)}+\|\ddg_3\|_{1,\al;\Ga_1^z}^{(-\al)}+\|\ddg_4\|_{1,\al;W_{1+}^z}\bigg),
\\
\end{split}
\end{equation}
with $\al\in(0,1)$. 
\end{lemma}

The proof follows by a straightforward adaptation of the argument for Lemma \ref{lem_p_th_dot+}, and is therefore omitted.

Utilizing \eqref{exp_ddS_sharp}, one has
\begin{equation}\label{est_sharp_dd_skaq}
\begin{split}
\|\ddS_\sharp\|_{1,\al;N_+^z}^{(-\al)}= \|\ddg_s\|_{1,\al;\Ga_s^z}^{(-\al)}
\end{split}
\end{equation}

Thus, $\ddeta_\sharp$ can be expressed as
\begin{equation*}
\ddeta_\sharp'=\frac{1}{[\bp]}(\beta_0^+\cdot\dd\mlU_\sharp-\ddg_0),
\end{equation*}
yielding the estimate
\begin{equation}\label{est_sharp_dd_eta}
\|\ddeta_\sharp'\|_{1,\al;\Ga_s^z}^{(-\al)}\leq C\left( \|\dd\mlU_\sharp(\bet^*,z_2)\|_{1,\al;\Ga_s^z}^{(-\al)}
+\|\ddg_0\|_{1,\al;\Ga_s^z}^{(-\al)}
\right).
\end{equation}

In summary, we state the following theorem.
\begin{theorem}\label{thm_sharp} 
For arbitrary $(\dd\mlU,\ddeta')\in\mr{\dN}_{\si^{3/2}}$ and $f'(\bet^*)\neq 0$, let $\ddeta^*$ be constructed in Lemma \ref{lem_nonlin_comp}. Then, the linear boundary value problem \eqref{exp_ddS_sharp}-\eqref{eq_sharp_dd_eta} admits a unique solution $(\dd\mlU_\sharp,\ddeta_\sharp')$ satisfying the following estimate
\begin{equation}\label{est_thm_sharp}
\begin{split}
\|\dd\mlU_\sharp\|_{1,\al;N_+^z}^{(-\al)}+
\|\ddeta_\sharp'\|_{1,\al;\Ga_s^z}^{(-\al)}\leq&  C\bigg(
\|\dd\mlF\|_{0,\al;N_+^z}^{(-\al+1)}
+\sum_{j=0}^1\|\ddg_j\|_{1,\al;\Ga_s^z}^{(-\al)}+\|\ddg_s\|_{1,\al;\Ga_s^z}^{(-\al)}\\
&+\|\ddg_3\|_{1,\al;\Ga_1^z}^{(-\al)}+\|\ddg_4\|_{1,\al;W_{1+}^z}\bigg),
\end{split}
\end{equation}
where the constant $C$ depends on $\al,L_0,L_1$ and the background flow.
\end{theorem}

\subsection{Contractive mapping} The analysis in the previous subsection suggests that we can define a mapping 
\begin{equation}\label{def_T}
\mr{T}:(\dd\mlU,\ddeta')\mapsto(\dd\mlU_\sharp,\ddeta_\sharp'),
\end{equation}
between suitable function spaces. We shall show that for sufficiently small $\si>0$, the operator $\mr{T}$ maps the space $\mr{\dN}_{\si^{3/2}}$ into itself.
\begin{lemma}
Under the same assumptions as in Theorem \ref{thm_sharp}, there exists a constant  $0<\si_3 <\si_2$  such that for all $\si\in(0,\si_3)$, the mapping $\mr{T}$ is well-defined on the space $\mr{\dN}_{\si^{3/2}}$.
\end{lemma}

\begin{proof}
Recalling that
\begin{equation*}
\dd\mlU_\sharp-\dot\mlU_+=(\dd\u_\sharp-\dot\u_+,\ddS_\sharp-\dS_+,0,0).
\end{equation*}

The difference of the entropy satisfies
\begin{equation}
\ddS_\sharp-\dS_+=\ddg_s-\dg_s.
\end{equation}

The boundary value problem for $\dd\u_\sharp-\dot\u_+$ reads:
\begin{equation}
\begin{cases}
\mlA_1(\bar\mlU_+)\p_{z_1}(\dd\u_\sharp-\dot\u_+) +\mlA_2(\bar\mlU_+)\p_{z_2} (\dd\u_\sharp-\dot\u_+)=\dd\mlF-\dot\mlF,\,\text{in }N_+^z,\\
\ddu_{\sharp,1}-\du_{1+}=\ddg_1-\dg_1,\,\text{on }\Ga_s^z,\\
\ddu_{\sharp,1}-\du_{1+}=\ddg_3-\dg_3,\,\text{on }\Ga_1^z,\\
\ddu_{\sharp,2}-\du_{2+}=\ddg_{2\i+2}-\dg_{2\i+2},\,\text{on }W_{\i+}^z,
\end{cases}
\end{equation}
where $\dot\mlF=(0,\dot{f}_2)$.

The difference in the shock front is given by
\begin{equation}
\ddeta'_\sharp-\deta'=\frac{1}{[\bp]}\lrb{\beta_0^+\cdot (\dd\mlU_\sharp-\dot\mlU_+)+\ddg_0-\dg_0}.
\end{equation}

It follows from Theorem \ref{thm_sharp} that
\begin{equation}\label{est_sharpU-dmlU}
\begin{split}
&\|\dd\mlU_\sharp-\dot\mlU_+\|_{1,\al;N_+^z}^{(-\al)}+
\|\ddeta_\sharp'-\deta'\|_{1,\al;\Ga_s^z}^{(-\al)}\\
\leq&  C\bigg(
\|\dd\mlF-\dot\mlF\|_{0,\al;N_+^z}^{(-\al+1)}
+\sum_{j=0}^1\|\ddg_j-\dg_j\|_{1,\al;\Ga_s^z}^{(-\al)}+\|\ddg_s-\dg_s\|_{1,\al;\Ga_s^z}^{(-\al)}\\
&+\|\ddg_3-\dg_3\|_{1,\al;\Ga_1^z}^{(-\al)}+\|\ddg_4-\dg_4\|_{1,\al;W_{1+}^z}\bigg),
\end{split}
\end{equation}

In what follows, we examine the magnitude of all terms coming from the right-hand side \eqref{est_sharpU-dmlU}.

It follows from \eqref{est_ddeta*} and \eqref{est_ddgs_ddg1} that
\begin{equation}
\begin{split}
(\ddg_s,\ddg_1)=(\dg_s,\dg_1)+O(1)\si^2.
\end{split}
\end{equation}

A direct calculus leads to
\begin{equation}
\begin{split}
\ddg_0-\dg_0=&\ddeta'[\bp]-\beta_0^+\cdot \dd\mlU+G_0(\mlU,\mlU_-(\eta(z_2),z_2))-\beta_0^-\cdot\dot\mlU_-(\bet^*,z_2)\\
=&\ddeta'[\bp]-\beta_0^+\cdot \dd\mlU -\beta_0^-\cdot \dd\mlU_-(\eta(z_2),z_2) +G_0(\mlU,\mlU_-(\eta(z_2),z_2))\\
&+\beta_0^-\cdot (\dd\mlU_-(\eta(z_2),z_2)-\dot\mlU_-(\eta(z_2),z_2))\\
&+\beta_0^-\cdot(\dot\mlU_-(\eta(z_2),z_2)-\dot\mlU_-(\bet^*+\ddeta^*,z_2))\\
&+\beta_0^-\cdot(\dot\mlU_-(\bet^*+\ddeta^*,z_2)-\dot\mlU_-(\bet^*,z_2)).
\end{split}
\end{equation}

It follows that
\begin{equation}\label{est_well_df_ddg_01s}
\sum_{j=0}^1\|\ddg_j-\dg_j\|_{1,\al;\Ga_s^z}^{(-\al)}+\|\ddg_s-\dg_s\|_{1,\al;\Ga_s^z}^{(-\al)}\leq C\si^2.
\end{equation}

We claim that
\begin{equation}\label{well_df_est_mlF}
\sum_{j=1}^{2}\|\dd\mlF-\dot\mlF\|_{0,\al;N_+^z}^{(-\al+1)} \leq C\si^2.
\end{equation}

Since the first component of $\dd\mlF-\dot\mlF$ is given by $\dd{f_2}$, the estimate \eqref{well_df_est_mlF} is a direct consequence of \eqref{expan_ddf2} and \eqref{est_ddeta*}.

The second component of $\dd\mlF-\dot\mlF$ is $\dd{f_2}-\dot{f}_2$, and its can be estimated as

\begin{equation}
\begin{split}
\|\dd{f}_2-\dot{f}_2\|_{0,\al;N_+^z}^{(-\al+1)} \leq &
C \|\ddg_s\|_{1,\al;N_+^z}^{(-\al)} \left\|\rho^\ga\frac{1+\ga \rho\ka^2 M_1^2}{\ga-1}-\brho_+^\ga\frac{1+\ga \brho_+\bka^2 \bM_+^2}{\ga-1}\right\|_{0,\al;N_+^z}^{(-\al+1)}\\
&+C  \|\ddg_s-\dg_s\|_{1,\al;N_+^z}^{(-\al)} \left\| \rho^\ga\frac{1+\ga \rho\ka^2 M_1^2}{\ga-1} \right\|_{0,\al;N_+^z}^{(-\al+1)}+O(1)\si^2\\
\leq& C \si^2.
\end{split}
\end{equation}

Observing that
\begin{equation}
\begin{split}
\ddg_3-\dg_3=-b_u^1 \du_{1-}(\bet^*+\ddeta^*,z_2)-b_u^1 \du_{1-}(\bet^*,z_2)+O(1)\si^2,
\end{split}
\end{equation}
we obtain
\begin{equation}\label{est_well_df_ddg_3}
\|\ddg_3-\dg_3\|_{1,\al;\Ga_1^z}^{(-\al)}\leq C \si^2.    
\end{equation}

Finally, it follows from
\begin{equation*}
\ddg_4-\dg_4=\si \bu_+\lrb{f'(z_1+\frac{L_1-z_1}{L_1-\bet^*}\ddeta^*)-f'(z_1)} 
+\si f'(z_1+\frac{L_1-z_1}{L_1-\bet^*}\ddeta^*) \ddu_1,
\end{equation*}
that
\begin{equation}\label{est_well_df_ddg_4}
\|\ddg_4-\dg_4\|_{1,\al;\Ga_1^z}^{(-\al)}\leq C \si^2.    
\end{equation}

Combining the estimates \eqref{est_well_df_ddg_01s}, \eqref{well_df_est_mlF}, \eqref{est_well_df_ddg_3} and \eqref{est_well_df_ddg_4}, we conclude from \eqref{est_sharpU-dmlU} that
\begin{equation}
\|\dd\mlU_\sharp-\dot\mlU_+\|_{1,\al;N_+^z}^{(-\al)}+
\|\ddeta_\sharp'-\deta'\|_{1,\al;\Ga_s^z}^{(-\al)}\leq C\si^2 = C\si^{\frac{1}{2}}\si^{\frac{3}{2}}.
\end{equation}

The desired result follows immediately from the inequality above.
\end{proof}

In addition, we show that the mapping $\mr{T}$ is contractive.

\begin{lemma}
There exists a constant $0<\si_4<\si_3$ such that for all $\si\in(0,\si_4)$, the mapping $\mr{T}$ defined in \eqref{def_T} is contractive in $\mr{\dN}_{\si^{3/2}}$.
\end{lemma}
\begin{proof}
Let $(\dd\mlU_j,\ddeta'_j)\in \mr{\dN}_{\si^{3/2}},\,j=1,2$ be given. Then Lemma \ref{lem_nonlin_comp} guarantees the unique existence of the corresponding $\ddeta^*_j,j=1,2$.

Consequently, for each $j=1,2$, there exists a unique solution $(\dd\mlU_{\sharp j},\ddeta_{\sharp j})$ to the boundary value problem \eqref{exp_ddS_sharp}-\eqref{eq_sharp_dd_eta}, where $(\dd\mlU,\ddeta',\ddeta^)$ are replaced by $(\dd\mlU_j,\ddeta'_j,\ddeta^*_j)$, respectively.

Then, the difference of the entropy satisfies
\begin{equation}\label{eq_sharp1-2_S}
\ddS_{\sharp1}-\ddS_{\sharp2}=\ddg_s^1-\ddg_s^2,
\end{equation}
the difference $\dd\u_{\sharp1}-\dd\u_{\sharp2}$ satisfies
\begin{equation}\label{eq_sharp1-2}
\begin{cases}
\mlA_1(\bar\mlU_+)\p_{z_1}(\dd\u_{\sharp1}-\dd\u_{\sharp2}) +\mlA_2(\dd\u_{\sharp1}-\dd\u_{\sharp2})=\dd\mlF^1-\dd\mlF^2,\,\text{in }N_+^z,\\
\ddu_{\sharp1,1}-\ddu_{\sharp2,1}=\ddg_1^1-\ddg_1^2,\,\text{on }\Ga_s^z,\\
\ddu_{\sharp1,1}-\ddu_{\sharp2,1}=\ddg_3^1-\ddg_3^2,\,\text{on }\Ga_1^z,\\
\ddu_{\sharp1,2}-\ddu_{\sharp2,2}=\ddg_{2\i+2}^1-\ddg_{2\i+2}^2,\,\text{on }W_{\i+}^z,
\end{cases}
\end{equation}
and
\begin{equation}\label{eq_sharp1-2_eta}
\ddeta'_{\sharp1}-\ddeta'_{\sharp2}=\frac{1}{[\bp]}\lrb{\beta_0^+\cdot (\dd\mlU_{\sharp1}-\dd\mlU_{\sharp2})+\ddg_0^1-\ddg_0^2}.
\end{equation}

The indices $j=1,2$ appearing as superscripts on the right-hand side of equations \eqref{eq_sharp1-2_S}-\eqref{eq_sharp1-2_eta} represent the substitution of the corresponding $(\dd\mlU_j,\ddeta'_j,\ddeta^*_j)$.

Using the estimate \eqref{est_thm_sharp} again yields
\begin{equation}\label{est_sharp1-2}
\begin{split}
&\|\dd\mlU_{\sharp1}-\dd\mlU_{\sharp2}\|_{1,\al;N_+^z}^{(-\al)}+
\|\ddeta_{\sharp1}'-\ddeta'_{\sharp2}\|_{1,\al;\Ga_s^z}^{(-\al)}\\
\leq&  C\bigg(
\|\dd\mlF^1-\dd\mlF^2\|_{0,\al;N_+^z}^{(-\al+1)}
+\sum_{j=0}^1\|\ddg_j^1-\ddg_j^2\|_{1,\al;\Ga_s^z}^{(-\al)}+\|\ddg_s^1-\ddg_s^2\|_{1,\al;\Ga_s^z}^{(-\al)}\\
&+\|\ddg_3^1-\ddg_3^2\|_{1,\al;\Ga_1^z}^{(-\al)}+\|\ddg_4^1-\ddg_4^2\|_{1,\al;W_{1+}^z}\bigg),
\end{split}
\end{equation}

We need to derive an estimate for $|\ddeta_1^*-\ddeta_2^*|$. By definition, we have
\begin{equation}\label{I1-I2}
\begin{split}
0&= \I(\dd\eta^*_1;\dd\mlU_1,\dd\eta'_1,\dd\mlU_-)- \I(\dd\eta^*_2;\dd\mlU_2,\dd\eta'_2,\dd\mlU_-)\\
&=(\ddeta_1^*-\ddeta_2^*)\int_0^1\frac{\p\I}{\p\ddeta^*}(\ddeta^*_\tau;\dd\mlU_2,\dd\eta'_2,\dd\mlU_-)\,d\tau\\
&+(\dd\mlU_1-\dd\mlU_2,\ddeta_1'-\ddeta_2')\cdot\int_0^1\nabla_{(\dd\mlU,\ddeta')}\I(\ddeta_1^*;\dd\mlU_\tau,\dd\eta'_\tau,\dd\mlU_-)\,d\tau,
\end{split}
\end{equation}
where $X_\tau=\tau X_1+(1-\tau)X_2$ for each $X=\ddeta^*,\dd\mlU,\ddeta'$.

One can derive a similar expansion to \eqref{expan_I} and consequently show that
\begin{equation*}
\begin{split}
\frac{\p\I}{\p\ddeta^*}(\ddeta^*_\tau;\dd\mlU_2,\dd\eta'_2,\dd\mlU_-)&=O(1)\si,\\ 
\nabla_{(\dd\mlU,\ddeta')}\I(\ddeta_1^*;\dd\mlU_\tau,\dd\eta'_\tau,\dd\mlU_-)&=O(1)\si,
\end{split}
\end{equation*}
where $O(1)$ is a constant depending on the background flow.

It follows from \eqref{I1-I2} that
\begin{equation*}
|\ddeta_1^*-\ddeta_2^*|\leq C\left(\|\dd\mlU_1-\dd\mlU_2\|_{1,\al;N_+^z}^{(-\al)}+
\|\ddeta_1'-\ddeta'_2\|_{1,\al;\Ga_s^z}^{(-\al)}\right).
\end{equation*}

This combining with \eqref{est_sharp1-2} yields
\begin{equation}\label{est_contrac}
\begin{split}
&\|\dd\mlU_{\sharp1}-\dd\mlU_{\sharp2}\|_{1,\al;N_+^z}^{(-\al)}+
\|\ddeta_{\sharp1}'-\ddeta'_{\sharp2}\|_{1,\al;\Ga_s^z}^{(-\al)}\\
\leq& C\si\left(
\|\dd\mlU_1-\dd\mlU_2\|_{1,\al;N_+^z}^{(-\al)}+
\|\ddeta_1'-\ddeta'_2\|_{1,\al;\Ga_s^z}^{(-\al)}
\right).  
\end{split}
\end{equation}

The contraction of $\mr{T}$ follows directly from \eqref{est_contrac} for sufficiently small $\si>0$, which completes the proof.
\end{proof}

\section*{Acknowledgment}
Weng is supported by National Natural Science Foundation of China 12571240, 12221001.

{\bf Data Availability Statement.} No data, models or code were generated or used during the study.

{\bf Conflict of interest.} On behalf of all authors, the corresponding author states that there is no conflict of interests.

\normalem
\bibliographystyle{siam}
\bibliography{RefMHD}

\end{document}